\documentclass[10pt]{article}

\usepackage{amsfonts}
\usepackage{array}
\usepackage{multirow}
\usepackage{amsmath}
\usepackage{amssymb}
\usepackage{xfrac}
\usepackage[dvipsnames]{xcolor}
\usepackage{amsthm}
\usepackage{fullpage}
\usepackage{float} 
\usepackage{stmaryrd}
\usepackage{enumerate}
\usepackage{mathtools}
\usepackage{tikz-cd}
\usepackage{bbm}
\usepackage[linktocpage]{hyperref}
\hypersetup{
    colorlinks=true,
    citecolor=red,
    linkcolor=blue,
    urlcolor=red,
}
\usepackage{newtxtext}
\usepackage[final]{showlabels}
\usepackage[capitalize, nameinlink]{cleveref}
\usepackage{tikz}
\usetikzlibrary{shapes.geometric}
\usepackage{comment}
\usepackage{mathrsfs}
\usepackage{fourier}
\usepackage{graphicx}
\usepackage[all]{xy}
\usepackage[titletoc]{appendix}
\usepackage{titlesec}
 \usepackage{etoolbox, chngcntr}
\AtBeginEnvironment{appendices}{%
 \titleformat{\section}{\bfseries\Large}{\appendixname~\thesection:}{0.5em}{}%
 \titleformat{\subsection}{\bfseries\large}{\thesubsection}{0.5em}{}%
\counterwithin{equation}{section}
}
\usepackage[nottoc]{tocbibind}
\usepackage{nicematrix}
\usepackage[backend=bibtex, giveninits=true, doi=false,isbn=false,url=false,style=alphabetic, maxnames=8, maxalphanames=9]{biblatex}
\bibliography{Ref.bib}
\DeclareMathAlphabet{\mathcal}{OMS}{cmsy}{m}{n}

\graphicspath{ {fig/} }

\definecolor{myyellow}{RGB}{200, 200, 0}

\newtheorem{theorem}{Theorem}
\numberwithin{theorem}{section}
\newtheorem{prop}[theorem]{Proposition}
\newtheorem{lemma}[theorem]{Lemma}
\newtheorem{corollary}[theorem]{Corollary}

\theoremstyle{definition}
\newtheorem{definition}[theorem]{Definition}
\newtheorem*{remark}{Remark}
\newtheorem{example}{Example}
\newtheorem{question}{Question}

\newtheorem*{Maintheorem}{Main Theorem}

\newtheorem{maintheorem}{Theorem}

\newcommand\dboxed[1]{
\raisebox{-1ex}{
\begin{tikzpicture}
         \node[draw,dashed] {\text{$#1$}}; 
   \end{tikzpicture}}
}

\newcommand\dbox[1]{
\begin{tikzpicture}
         \node[draw,dashed] {\text{$#1$}}; 
   \end{tikzpicture}
}
\newcommand{\igc}[2]{\begin{center} \includegraphics[scale=#1]{fig/#2} \end{center}}

\newcommand{\ext}{\mathrm{Ext}}
\newcommand{\bs}{\mathrm{BS}}
\newcommand{\kb}{\mathbbm{k}}
\renewcommand{\hom}{\mathrm{Hom}}
\newcommand{\im}{\mathrm{im } \,}

\newcommand{\Db}{\mathbb{D}}

\newcommand{\sbim}{\mathbb{S}\textnormal{Bim}}

\newcommand{\hh}{\mathrm{HH}}
\newcommand{\hhh}{\mathrm{HHH}}

\newcommand\scalemath[2]{\scalebox{#1}{\mbox{\ensuremath{\displaystyle #2}}}}
\newcommand{\sbrac}[1]{\left[#1 \right]}
\newcommand{\abrac}[1]{\left\langle#1\right\rangle}

\newcommand{\paren}[1]{\left( #1 \right)}
\newcommand{\set}[1]{\left \{ #1 \right \}}

\newcommand{\s}[1]{\scalemath{0.8}{#1}}
\newcommand{\un}[1]{\underline{#1}}

\newcommand{\Zb}{\mathbb{Z}}

\makeatletter
\DeclareRobustCommand\widecheck[1]{{\mathpalette\@widecheck{#1}}}
\def\@widecheck#1#2{%
    \setbox\z@\hbox{\m@th$#1#2$}%
    \setbox\tw@\hbox{\m@th$#1%
       \widehat{%
          \vrule\@width\z@\@height\ht\z@
          \vrule\@height\z@\@width\wd\z@}$}%
    \dp\tw@-\ht\z@
    \@tempdima\ht\z@ \advance\@tempdima2\ht\tw@ \divide\@tempdima\thr@@
    \setbox\tw@\hbox{%
       \raise\@tempdima\hbox{\scalebox{1}[-1]{\lower\@tempdima\box
\tw@}}}%
    {\ooalign{\box\tw@ \cr \box\z@}}}
\makeatother

\tikzset{
uni/.style={circle,fill,draw,inner sep=0mm,minimum size=1mm},
  midarrow/.style={postaction={decorate,decoration={markings,mark=at position #1 with {\arrow{>}}}}},
  midarrow/.default=0.5,
  midarrowrev/.style={postaction={decorate,decoration={markings,mark=at position #1 with {\arrow{<}}}}},
  midarrowrev/.default=0.5,
  dot/.style={circle,fill,draw,inner sep=0mm,minimum size=1.3mm},  
  rdot/.style={circle,fill, color=red, draw,inner sep=0mm,minimum size=1.3mm},  
  bdot/.style={circle,fill, color=blue, draw,inner sep=0mm,minimum size=1.3mm},  
  pdot/.style={circle,color=Violet,fill=Violet,draw,inner sep=0mm,minimum size=1.3mm},  
  halfdot/.style={circle,fill=black, opacity=0.3 ,draw,inner sep=0mm,minimum size=1.1mm}, 
  hdot/.style={circle,fill=white,draw,inner sep=0mm,minimum size=1.3mm}, 
  rhdot/.style={circle,color=red,fill=white,draw,inner sep=0mm,minimum size=1.3mm, }, 
  bhdot/.style={circle,color=blue,fill=white,draw,inner sep=0mm,minimum size=1.3mm, },
  ghdot/.style={circle,color=teal,fill=white,draw,inner sep=0mm,minimum size=1.3mm, },
  yhdot/.style={circle,color=myyellow,fill=white,draw,inner sep=0mm,minimum size=1.3mm, }, 
  rkhdot/.style={circle,color=red,fill=white,draw,inner sep=0mm,minimum size=4mm, }, 
  rbkhdot/.style={circle,color=red,fill=white,draw,inner sep=0mm,minimum size=6mm, }, 
  rbbkhdot/.style={circle,color=red,fill=white,draw,inner sep=0mm,minimum size=6.7mm, }, 
  btridot/.style={rectangle,color=blue,fill=white,draw,inner sep=0.3mm,minimum size=1.7mm}, 
    rtridot/.style={rectangle,color=red,fill=white,draw,inner sep=0.3mm,minimum size=1.7mm}, 
  every picture/.style=thick
}
\newsavebox\lowerdot
\savebox\lowerdot{%
\begin{tikzpicture}[scale=0.3,thick,baseline]
 \draw (0,-0.5) to (0,0.5);
 \node at (0,-0.5) {$\bullet$};
\end{tikzpicture}%
}
\newsavebox\upperdot
\savebox\upperdot{%
\begin{tikzpicture}[scale=0.3,thick,baseline]
 \draw (0,-0.5) to (0,0.5);
 \node at (0,0.5) {$\bullet$};
\end{tikzpicture}%
}

\newcommand{\rzero}{\ \raisebox{-2ex}{\begin{tikzpicture}[xscale=0.25,yscale=0.3,thick,baseline]
 \draw[red] (0,0) -- (0,2); 
\end{tikzpicture}}}

\newcommand{\bzero}{\ \raisebox{-2ex}{\begin{tikzpicture}[xscale=0.25,yscale=0.3,thick,baseline]
 \draw[blue] (0,0) -- (0,2); 
\end{tikzpicture}}}

\newcommand{\ds}{\ \raisebox{-2ex}{\begin{tikzpicture}[xscale=0.25,yscale=0.3,thick,baseline]
 \draw[red] (0,-0.1) -- (0,2.1);
 \node[rtridot] at (0,1) {$\scalemath{0.9}{d_s}$};
\end{tikzpicture}}}

\newcommand{\dt}{\ \raisebox{-2ex}{\begin{tikzpicture}[xscale=0.25,yscale=0.3,thick,baseline]
 \draw[blue] (0,0) -- (0,2); 
 \node[btridot] at (0,1) {$\scalemath{0.9}{d_t}$};
\end{tikzpicture}}}

\newcommand{\rone}{\ \raisebox{-2ex}{\begin{tikzpicture}[xscale=0.25,yscale=0.3,thick,baseline]
 \draw[red] (0,0) -- (0,1);
 \draw[red] (0,2) -- (0,3);
 \node[dot, red] at (0,1) {};
  \node[dot, red] at (0,2) {};
\end{tikzpicture}}}

\newcommand{\bone}{\ \raisebox{-2ex}{\begin{tikzpicture}[xscale=0.25,yscale=0.3,thick,baseline]
 \draw[blue] (0,0) -- (0,1);
 \draw[blue] (0,2) -- (0,3);
 \node[dot, blue] at (0,1) {};
  \node[dot, blue] at (0,2) {};
\end{tikzpicture}}}

\newcommand{\rhunit}{\ \raisebox{0ex}{\begin{tikzpicture}[xscale=0.25,yscale=0.3,thick,baseline]
 \draw[red] (0,0) -- (0,1);
 \node[rhdot] at (0,0) {};
\end{tikzpicture}}}

\newcommand{\runit}{\ \raisebox{0ex}{\begin{tikzpicture}[xscale=0.25,yscale=0.3,thick,baseline]
 \draw[red] (0,0) -- (0,1);
 \node[dot, red] at (0,0) {};
\end{tikzpicture}}}

\newcommand{\bhunit}{\ \raisebox{0ex}{\begin{tikzpicture}[xscale=0.25,yscale=0.3,thick,baseline]
 \draw[blue] (0,0) -- (0,1);
 \node[bhdot] at (0,0) {};
\end{tikzpicture}}}

\newcommand{\gunit}{\ \raisebox{0ex}{\begin{tikzpicture}[xscale=0.25,yscale=0.3,thick,baseline]
 \draw[teal] (0,0) -- (0,1);
 \node[dot, teal] at (0,0) {};
\end{tikzpicture}}}

\newcommand{\ghunit}{\ \raisebox{0ex}{\begin{tikzpicture}[xscale=0.25,yscale=0.3,thick,baseline]
 \draw[teal] (0,0) -- (0,1);
 \node[ghdot] at (0,0) {};
\end{tikzpicture}}}

\newcommand{\rhzero}{\ \raisebox{0ex}{\begin{tikzpicture}[xscale=0.25,yscale=0.4,thick,baseline]
 \draw[red] (0,-1) -- (0,1);
 \node[rhdot] at (0,0) {};
\end{tikzpicture}}}

\newcommand{\bhzero}{\ \raisebox{0ex}{\begin{tikzpicture}[xscale=0.25,yscale=0.4,thick,baseline]
 \draw[blue] (0,-1) -- (0,1);
 \node[bhdot] at (0,0) {};
\end{tikzpicture}}}

\newcommand{\bunit}{\ \raisebox{0ex}{\begin{tikzpicture}[xscale=0.25,yscale=0.3,thick,baseline]
 \draw[blue] (0,0) -- (0,1);
 \node[dot, blue] at (0,0) {};
\end{tikzpicture}}}

\newcommand{\counit}{\ \raisebox{1ex}{\begin{tikzpicture}[xscale=0.25,yscale=0.3,thick,baseline]
 \draw (0,-1) -- (0,0);
 \node[dot] at (0,0) {};
\end{tikzpicture}}}

\newcommand{\rcounit}{\ \raisebox{0ex}{\begin{tikzpicture}[xscale=0.25,yscale=0.3,thick,baseline]
 \draw[red] (0,-1) -- (0,0);
 \node[dot, red] at (0,0) {};
\end{tikzpicture}}}

\newcommand{\bcounit}{\ \raisebox{0ex}{\begin{tikzpicture}[xscale=0.25,yscale=0.3,thick,baseline]
 \draw[blue] (0,-1) -- (0,0);
 \node[dot, blue] at (0,0) {};
\end{tikzpicture}}}

\newcommand{\bhcounit}{\ \raisebox{0ex}{\begin{tikzpicture}[xscale=0.25,yscale=0.3,thick,baseline]
 \draw[blue] (0,-1) -- (0,0);
 \node[bhdot] at (0,0) {};
\end{tikzpicture}}}

\newcommand{\lcounit}[1]{ \raisebox{0ex}{\begin{tikzpicture}[xscale=0.25,yscale=0.3,thick,baseline]
\node at (0,-1.6) {$#1$};
 \draw (0,-1) -- (0,0);
 \node[dot] at (0,0) {};
\end{tikzpicture}} \hspace{-1ex}}

\newcommand{\lunit}[1]{ \raisebox{0ex}{\begin{tikzpicture}[xscale=0.25,yscale=0.3,thick,baseline]
 \draw (0,0) -- (0,1.2);
 \node[dot] at (0,0) {};
 \node at (0,-1) {$#1$};
\end{tikzpicture}}}

\newcommand{\lhunit}[1]{ \raisebox{0ex}{\begin{tikzpicture}[xscale=0.25,yscale=0.3,thick,baseline]
 \draw (0,0) -- (0,1.2);
 \node[hdot] at (0,0) {};
 \node at (0,-1) {$#1$};
\end{tikzpicture}} }

\newcommand{\lzero}[1]{ \raisebox{-1.7ex}{\begin{tikzpicture}[xscale=0.25,yscale=0.3,thick,baseline]
 \draw (0,0) -- (0,2);
 \node at (0,-0.7) {$#1$};
\end{tikzpicture} } \hspace{-1ex}}

\newcommand{\lhalfunit}[1]{ \raisebox{-3.5ex}{\begin{tikzpicture}[scale=0.6]
    \draw (0,0.1)-- (0,0.6);
    \draw[thick] (0,0) circle (1.3mm);
    
    \fill[black] (0,0.15) arc[start angle=90, end angle=270, radius=1.3mm] -- cycle;
    \node at (0,-0.5) {$#1$};
\end{tikzpicture}}
}

\newcommand{\rpitchfork}{\ \raisebox{-2ex}{\begin{tikzpicture}[xscale=0.25,yscale=0.3,thick,baseline]
             \draw[red] (1,1.3) to (1,2.1);
             \draw[red] (0,0) to[out=90,in=-180] (1,1.3);
	       \draw[red] (2,0) to[out=90,in=0] (1,1.3);
	       \node[bdot] at (1,0.8) {};
        \draw[blue] (1,0) to (1,0.8);
\end{tikzpicture}}}

\newcommand{\bpitchfork}{\ \raisebox{-2ex}{\begin{tikzpicture}[xscale=0.25,yscale=0.3,thick,baseline]
             \draw[blue] (1,1.3) to (1,2.1);
             \draw[blue] (0,0) to[out=90,in=-180] (1,1.3);
	       \draw[blue] (2,0) to[out=90,in=0] (1,1.3);
	       \node[rdot] at (1,0.8) {};
        \draw[red] (1,0) to (1,0.8);
\end{tikzpicture}}}

\newcommand{\rcup}{\ \raisebox{2ex}{\begin{tikzpicture}[xscale=0.25,yscale=-0.3,thick,baseline]
             \draw[red] (0,0) to[out=90,in=-180] (1,1.3);
	       \draw[red] (2,0) to[out=90,in=0] (1,1.3);
\end{tikzpicture}}}

\newcommand{\rpitchcup}{\ \raisebox{2ex}{\begin{tikzpicture}[xscale=0.25,yscale=-0.3,thick,baseline]
             \draw[red] (0,0) to[out=90,in=-180] (1,1.3);
	       \draw[red] (2,0) to[out=90,in=0] (1,1.3);
	       \node[bdot] at (1,0.7) {};
        \draw[blue] (1,0) to (1,0.7);
\end{tikzpicture}}}

\newcommand{\rhpitchcupin}{\ \raisebox{2ex}{\begin{tikzpicture}[xscale=0.25,yscale=-0.3,thick,baseline]
             \draw[red] (0,0) to[out=90,in=-180] (1,1.3);
	       \draw[red] (2,0) to[out=90,in=0] (1,1.3);
              \draw[blue] (1,0) to (1,0.7);
              \node[bhdot] at (1,0.7) {};
\end{tikzpicture}}}

\newcommand{\rhpitchcupout}{\ \raisebox{2ex}{\begin{tikzpicture}[xscale=0.25,yscale=-0.3,thick,baseline]
             \draw[red] (0,0) to[out=90,in=-180] (1,1.3);
	       \draw[red] (2,0) to[out=90,in=0] (1,1.3);
              \node[rhdot] at (1,1.3) {};
	       \node[bdot] at (1,0.7) {};
        \draw[blue] (1,0) to (1,0.7);
\end{tikzpicture}}}

\newcommand{\rhhpitchcupout}{\ \raisebox{2ex}{\begin{tikzpicture}[xscale=0.25,yscale=-0.3,thick,baseline]
             \draw[red] (0,0) to[out=90,in=-180] (1,1.3);
	       \draw[red] (2,0) to[out=90,in=0] (1,1.3);
              \node[rhdot] at (1,1.3) {};
             \draw[blue] (1,0) to (1,0.7);
	       \node[bhdot] at (1,0.7) {};
\end{tikzpicture}}}
\newcommand{\bpitchcup}{\ \raisebox{2ex}{\begin{tikzpicture}[xscale=0.25,yscale=-0.3,thick,baseline]
             \draw[blue] (0,0) to[out=90,in=-180] (1,1.3);
	       \draw[blue] (2,0) to[out=90,in=0] (1,1.3);
	       \node[rdot] at (1,0.7) {};
        \draw[red] (1,0) to (1,0.7);
\end{tikzpicture}}}

\newcommand{\bhpitchcupin}{\ \raisebox{2ex}{\begin{tikzpicture}[xscale=0.25,yscale=-0.3,thick,baseline]
             \draw[blue] (0,0) to[out=90,in=-180] (1,1.3);
	       \draw[blue] (2,0) to[out=90,in=0] (1,1.3);
              \draw[red] (1,0) to (1,0.7);
              \node[rhdot] at (1,0.7) {};
\end{tikzpicture}}}

\newcommand{\bhpitchcupout}{\ \raisebox{2ex}{\begin{tikzpicture}[xscale=0.25,yscale=-0.3,thick,baseline]
             \draw[blue] (0,0) to[out=90,in=-180] (1,1.3);
	       \draw[blue] (2,0) to[out=90,in=0] (1,1.3);
              \node[bhdot] at (1,1.3) {};
	       \node[rdot] at (1,0.7) {};
        \draw[red] (1,0) to (1,0.7);
\end{tikzpicture}}}

\newcommand{\bhhpitchcupout}{\ \raisebox{2ex}{\begin{tikzpicture}[xscale=0.25,yscale=-0.3,thick,baseline]
             \draw[blue] (0,0) to[out=90,in=-180] (1,1.3);
	       \draw[blue] (2,0) to[out=90,in=0] (1,1.3);
              \node[bhdot] at (1,1.3) {};
             \draw[red] (1,0) to (1,0.7);
	       \node[rhdot] at (1,0.7) {};
\end{tikzpicture}}}

\newcommand{\gpitchcup}{\ \raisebox{2ex}{\begin{tikzpicture}[xscale=0.25,yscale=-0.3,thick,baseline]
             \draw[red] (0,0) to[out=90,in=-180] (1,1.3);
	       \draw[red] (2,0) to[out=90,in=0] (1,1.3);
	       \node[rdot, teal] at (1,0.7) {};
        \draw[teal] (1,0) to (1,0.7);
\end{tikzpicture}}}

\newcommand{\ghpitchcupout}{\ \raisebox{2ex}{\begin{tikzpicture}[xscale=0.25,yscale=-0.3,thick,baseline]
             \draw[red] (0,0) to[out=90,in=-180] (1,1.3);
	       \draw[red] (2,0) to[out=90,in=0] (1,1.3);
            \draw[teal] (1,0) to (1,0.7);
	       \node[ghdot] at (1,0.7) {};
\end{tikzpicture}}}

\newcommand{\grhpitchcupout}{\ \raisebox{2ex}{\begin{tikzpicture}[xscale=0.25,yscale=-0.3,thick,baseline]
             \draw[red] (0,0) to[out=90,in=-180] (1,1.3);
	       \draw[red] (2,0) to[out=90,in=0] (1,1.3);
              \node[rhdot] at (1,1.3) {};
	       \node[dot, teal] at (1,0.7) {};
        \draw[teal] (1,0) to (1,0.7);
\end{tikzpicture}}}

\newcommand{\grhhpitchcupout}{\ \raisebox{2ex}{\begin{tikzpicture}[xscale=0.25,yscale=-0.3,thick,baseline]
             \draw[red] (0,0) to[out=90,in=-180] (1,1.3);
	       \draw[red] (2,0) to[out=90,in=0] (1,1.3);
              \node[rhdot] at (1,1.3) {};
             \draw[teal] (1,0) to (1,0.7);
	       \node[ghdot] at (1,0.7) {};
\end{tikzpicture}}}

\newcommand{\bgrhpitchcupout}{\ \raisebox{2ex}{\begin{tikzpicture}[xscale=0.25,yscale=-0.3,thick,baseline]
             \draw[blue] (0,0) to[out=90,in=-180] (1,1.3);
	       \draw[blue] (2,0) to[out=90,in=0] (1,1.3);
              \node[bhdot] at (1,1.3) {};
	       \node[dot, teal] at (1,0.7) {};
        \draw[teal] (1,0) to (1,0.7);
\end{tikzpicture}}}

\newcommand{\0}{\ \raisebox{-1ex}{\begin{tikzpicture}[xscale=0.25,yscale=0.3,thick,baseline]
 \draw (0,0) -- (0,2); 
\end{tikzpicture}}}

\usepackage{braids}

\title{\textbf{The Last Three T-degrees in Triply-Graded Link Homology}}
\date{\relax }
\author{Cailan Li}

\begin{document}

\maketitle
\begin{abstract}
 We investigate the structure of reduced triply graded link homology $\overline{\hhh}$ in the top/bottom three $T-$degrees for links arising as closures of positive/negative braids. Using a diagrammatic approach to the Hochschild cohomology of Soergel bimodules, we provide explicit computations of $\overline{\hhh}$ as $R-$modules in these degrees. Our results reveal that the homology here is often zero, especially in the negative braid case, and display striking uniformity.
\end{abstract}

\tableofcontents


\section{Introduction}

Link homology theories provide powerful link invariants that refine classical link polynomials. The field began with Khovanov’s categorification of the Jones polynomial $J(L)$ \cite{Kho00}, known today as Khovanov homology. The HOMFLY polynomial is a generalization of the Jones polynomial and in \cite{Kho07}, Khovanov constructed a triply-graded link homology $\overline{\hhh}$\footnote{This is reduced $\hhh$ and all computations in this paper are for the reduced theory.} which categorifies the HOMFLY polynomial. The HOMFLY polynomial can be constructed as a two-variable trace on the type $A$ Hecke algebra and it is through this lens that Khovanov's construction categorifies.

\subsection{Background}

The type $A$ Hecke algebra $\mathbf{H}_{S_n}$, being a quotient of the braid group $B_n$, has a standard presentation with generators $\set{\sigma_s}_{s\in S}$ where $S=\set{1,\ldots,n-1 }$. It also has a Kazhdan-Lusztig presentation over $\mathbb{Z}[v^{\pm 1}]$ with generators $\set{b_s}_{s\in S}$. In \cite{Soe90}, Soergel constructs $\sbim_n$, a certain full graded subcategory of $R-$bimodules where $R=\kb[\alpha_1, \ldots, \alpha_{n-1}]$ with grading $|\alpha_i|=2$. Known today as (type A) Soergel Bimodules, $\sbim_n$ is monoidally generated via $\otimes_R$ by objects $\set{B_s}_{s\in S}$ and Soergel shows that the split Grothendieck group $K_\oplus(\sbim_n)$ is isomorphic to $\mathbf{H}_{S_n}$. $\sbim_n$ is idempotent complete, which yields an isomorphism $K_\oplus(\sbim_n)\cong K_\Delta(K^b(\sbim_n))$. This is the starting point for Khovanov's construction of $\overline{\hhh}$ which proceeds as follows.

\begin{itemize}
\item Given braid $\beta\in B_n\leadsto$ Rouquier complex 
 $F^\bullet_{\beta}   \in K^b(\sbim_n)$ \cite{Rou04} defined on generators
 $$F^\bullet_{\sigma_s}=\un{B_s}\xrightarrow{\rcounit}R(1), \qquad F^\bullet_{\sigma_s^{-1}}=R(-1)\xrightarrow{\runit} \un{B_s}$$
 and extended via $F^\bullet_{\beta \beta^\prime}=F^\bullet_{\beta}\otimes_R F^\bullet_{\beta^\prime}$. The underline indicates cohomological degree 0.
 \item For each $k\in \mathbb{Z}^+$, apply the Hochschild cohomology functor $\hh^k(-):=\ext_{R^e}^{k}(R,-)$ to $F^\bullet_{\beta} $ \textit{termwise}. Get complex \vspace{-1ex}
        \[ \hh^k(F^\bullet_\beta)= \ [ \ \ldots\to \hh^k( F^i_{\beta}) \xrightarrow{\hh^k(d_{F^\bullet_{\beta} })} \hh^k( F^{i+1}_{\beta})   \to \ldots \vspace{-2ex} \ ]  \]
\item Taking cohomology $\leadsto$ vector space ($R-$module) \vspace{0.5ex}$\overline{\hhh}(\beta)$ with 3 gradings:  $A=$ Hochschild degree (which $k$) , $T=$ homological degree (which $i$) , $Q=$ grading from $R$. Setting $T=-1$ recovers the HOMFLY polynomial of $\widehat{\beta}$ after a grading shift . 
 \end{itemize}

Unlike Khovanov homology, $\overline{\hhh}(\beta)$ is substantially more difficult to compute explicitly. As seen above, it's a homology of homology construction, and so working with the terms of the complex and the differential becomes unwieldy fast. As an illustration, the smallest nontrivial link is the Hopf link, $H=\widehat{\sigma_1^2}$. Here $n=2$, $R=\kb[\alpha_1]$ and $B_s=R\otimes_{\kb[\alpha_1^2]}R(1)$. When computing the $A=1$ part of $\overline{\hhh}(\sigma_1^2)$, chain maps, chain homotopies, and their interaction with the differential will present immediate technical challenges. This underscores the computational challenges inherent to $\overline{\hhh}$ even in minimal examples. The answer, using diagrammatics can be found in \cite[Appendix B]{Li22}. In contrast, the complex computing Khovanov homology of $H$ is just
\[ \underline{\Zb(1)\oplus \Zb(-1) \oplus \Zb(-1) \oplus \Zb(-3)} \xrightarrow{\begin{pmatrix}
    1 & 0 & 0 & 0 \\
    0 & 1 & 1 & 0 \\
    1 & 0 & 0 & 0 \\
    0 & 1 & 1 & 0
\end{pmatrix}} \paren{\Zb(1)\oplus \Zb(-1)} \oplus \paren{\Zb(1)\oplus \Zb(-1)} \xrightarrow{\begin{pmatrix}
    0 & 0 & 0 & 0 \\
    1 & 0 & -1 & 0 \\
    1 & 0 & -1 & 0 \\
    0 & 1 & 0 &-1
\end{pmatrix}}\Zb(3)\oplus \Zb(1) \oplus \Zb(1) \oplus \Zb(-1) \]
which one obtains directly from the definition. We summarize the differences between the theories in the table below. \vspace{-1ex}
{\renewcommand{\arraystretch}{1.2} 

\begin{table}[H]
\centering
\begin{tabular}{c|c|c}
 & Khovanov & $\overline{\hhh}$  \\\hline
\rule{0pt}{2ex}Terms of complex & $\otimes$ of $\Zb[x]/(x^2)$ &  $\ext^k_{R^e}(R,B_iB_j\ldots)=?$ \\\hline 
\multirow{2}{*}{Differential }& \multirow{2}{*}{from Frobenius algebra} &  Yoneda product$=?$\\
& &  $\ext^k(R, F^i_{\beta} )\times \hom(F^i_{\beta}, F^{i+1}_{\beta})\to\ext^k(R, F_\beta^{i+1})$
\end{tabular}
\caption{The $?$ reflects the fact that, in general, we lack an explicit description of these objects. }
\end{table}
\vspace{-2ex}
 
} 
Let $T(m,n)=(\sigma_1 \ldots \sigma_{m-1})^n$. In \cite{EH19} Elias and Hogancamp made a breakthrough   by developing a method to compute $\overline{\hhh}(T(n,n))$ where $n\ge 0$. Since then there have been other advancements, such as
 \begin{itemize}
 \item Using \cite{EH19}, \cite{HM19} computed  $\overline{\hhh}(T(m,n)) \ \forall m,n\ge 0$ via a recursion. 
 \item Using \cite{Ras15} reformulation of $\overline{\hhh}$, \cite{NS24} wrote an algorithm to compute $\overline{\hhh}(K_{\le 11})$ where $K_{\le 11}$ is a \underline{knot} with $\le 11$ crossings.
 \end{itemize}
 However, the methods above have limitations. Specifically,
\begin{itemize}
    \item \cite{HM19} needs $\overline{\hhh}(T(m,n))$ to be parity and all intermediate steps of the recursion to be parity. This is seldomly observed. For $m,n$ not coprime, $\overline{\hhh}$ for the negative torus link $T(m,-n)$ is not parity. 
    \item \cite{NS24} needs $\overline{\hhh}(\beta)$ to be a f.d. vector space for their algorithm to terminate. When $\widehat{\beta}=L$ is a link $\overline{\hhh}(\beta)$ is not f.d. Memory issues also constrain the calculations to $\le 11$ crossings.
\end{itemize}
For example, the 6 crossing link $\overline{\hhh}(T(3,-3))$ does not satisfy either condition above and was only just recently computed in \cite[Appendix B]{Li22}. The approach involved developing a diagrammatic calculus for the right-hand side of the table (for rank 2 Soergel bimodules), enabling explicit computations of $\overline{\hhh}$ in specific cases. In particular this method circumvents the limitations above. Such a diagrammatic calculus was developed for the $k=0$ case in the works of Elias, Khovanov, and Williamson \cite{EK10}, \cite{DC}, \cite{SC} and has found remarkable applications in representation theory in the past decade or so. See \cite{HodgeSoergel}, \cite{Wtorsionexplod}, \cite{KoszulAMRW},\cite{LW22}, \cite{EQ23}, \cite{BCH23} for instance. \\

In \cite{M}, Makisumi initiated the diagrammatic study of higher Ext groups between Soergel Bimodules in the hopes of finding a triply graded analogue of the Koszul duality in \cite{KoszulAMRW}. \cite{Li22} builds upon \cite{M} and gave the first novel diagrammatic computations of $\overline{\hhh}$ for certain $3-$strand braids. In this paper, we further illustrate the power of the diagrammatics in \cite{Li22} by computing $\overline{\hhh}$ of any positive or negative $n$‑strand braid in prescribed $T$‑degrees.

\subsection{Main Results}
Let $|\beta|$ be the number of terms in the expression $\beta\in B_n$. From the definition it is clear that for a positive braid $\beta\in B_n^+$, $\overline{\hhh}(\beta)^{T=i}$ is supported in degrees $0\le i\le |\beta|$; similarly, for a negative braid $\alpha\in B_n^-$, $\overline{\hhh}(\alpha)^{T=i}$ is supported in degrees $-|\alpha|\le i\le 0$. This paper establishes the following

\begin{Maintheorem}
    We compute $\overline{\hhh}^{A,T, Q}(\beta)$ as an $R-$module in the three highest (lowest) $T-$degrees for any positive (negative) braid $\beta\in B_n$ for all values of $ A, Q$.
\end{Maintheorem}
This result combines \cref{mainpostheorem}, \cref{mainnegtheorem}, \cref{mainpostheorem2}, \cref{mainnegtheorem2}, \cref{theorem:hhhneg2} and the results in \cref{clossect}. The complex computing $\overline{\hhh}$ for a negative link is dual as an $R-$module to the complex for the corresponding positive link (see \cref{lem:dualcomplex}). However, previous computations of $\overline{\hhh}$ were only able to compute $\overline{\hhh}$ as a vector space. Because our methods compute $\overline{\hhh}$ as an $R-$module, this enables us to establish results for negative braids via \cref{lem:dualcomplex}. Our results are orthogonal to those of \cite{HM19}, as seen in the figure \vspace{-2ex}below\footnote{$\overline{\hhh}$ of negative torus knots is just a regrading of $\overline{\hhh}$ for positive torus knots}. 
\begin{figure}[H]
\centering
\begin{tikzpicture}[scale=0.5,>=stealth]
  \draw[->] (-0.5, 0) -- (5, 0) node[right] {Positive Braids};
  \draw[->] (0, -0.5) -- (0, 5) node[above] {$T-$degree};


  \foreach \y in {3, 3.7, 4.4}{
    \draw[red] (-0.1, \y) -- (5, \y); }
    
  \draw[thick, teal] (2.8, 0) -- (2.8, 5.2); 
 \node[left] at (-0.1,4.4) {$|\beta|\hspace{0.5ex}\phantom{-0}$};
 \node[left] at (-0.1,3.7) {$|\beta|-1$};
 \node[left] at (-0.1,3) {$|\beta|-2$};
 \node at (2.6,-0.5) {$T(m,n) \ m,n \ge 0$};
  \foreach \x in {2.3, 2.32,..., 2.7}
  {
    \draw[teal,opacity=0.3] (\x, 0) -- (\x, 5.2);
  }

  \foreach \x in {2.9, 2.92,..., 3.3}
  {
    \draw[teal,opacity=0.3] (\x, 0) -- (\x, 5.2);
  }
\end{tikzpicture}\qquad \qquad
\begin{tikzpicture}[scale=0.5,>=stealth]
  \draw[->] (-0.5, 0) -- (5.2, 0) node[right] {Negative Braids};
  \draw[->] (0, 0.5) -- (0, -5) node[below] {$T-$degree};


  \foreach \y in {-3, -3.7, -4.4}{
    \draw[red] (-0.1, \y) -- (5, \y); }
    
  \draw[thick, teal] (2.8, 0) -- (2.8, -5.2); 
 \node[left] at (-0.1,-4.4) { $-|\alpha|\hspace{0.5ex}\phantom{+2} $};
 \node[left] at (-0.1,-3.7) {$-|\alpha|+1$};
 \node[left] at (-0.1,-3) {$-|\alpha|+2$ };
 \node at (2.8,0.6) {$\substack{T(m,-n) \ m,n \ge 0 \\ (m,n)=1 }$};
  \foreach \x in {2.6, 2.62,..., 2.8}
  {
    \draw[teal,opacity=0.3] (\x, 0) -- (\x, -5.2);
  }

  \foreach \x in {2.8, 2.82,..., 3}
  {
    \draw[teal,opacity=0.3] (\x, 0) -- (\x, -5.2);
  }
\end{tikzpicture}
\caption{The green area represents the results of \cite{HM19} while the red lines show our new results.} \vspace{-2ex}
\end{figure}
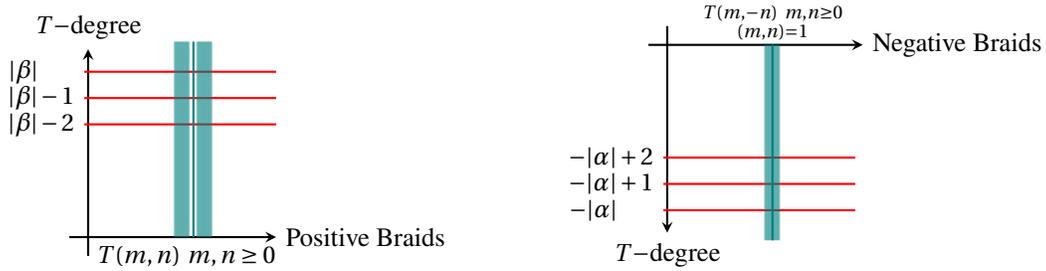
As previously alluded, our results show that $\overline{\hhh}$ in these $T-$degrees is rather uniform. Provided the link is ``suitably connected", $\overline{\hhh}(\beta)$ is independent of the choice of $\beta$ in these degrees. However there is an asymmetry for positive and negative braids at $T-$degree $|\beta|-2$. For positive braids, \cref{theorem:hhhpos2} and \cref{mainpostheorem2} have the same conclusion. For negative braids $\alpha\in B_3^-$, $\overline{\hhh}(\alpha)$ is nonzero at $A=2$ (see \cref{theorem:hhhneg2}), whereas for $\alpha\in B_n^-$, $n\ge 4$, $\overline{\hhh}(\alpha)$ vanishes in all $A-$degrees (see \cref{mainnegtheorem2}). This asymmetry is most likely related to the asymmetry of the Markov Move II for $\overline{\hhh}$. See \cite[3.1(c), 3.1(d)]{Hog18} for a pictorial description. 

\subsubsection{Applications}

Given a braid $\beta$, to obtain a link invariant for the braid closure $L=\widehat{\beta}$ one needs to normalize $\overline{\hhh}(\beta)$. Following \cite{HM19}, the normalization is given by

\begin{definition}
\label{def1}
    For $\beta\in B_n$, let $e=n_+-n_-$ where $n_+$($n_-$) is the number of positive(negative) crossings in $\beta$. Let $L=\widehat{\beta}$ and $c$ the number of components of $L$. The reduced superpolynomial of $\mathscr{P}_L$ is
    \[ \mathscr{P}_L= (Q^{-4}AT)^{\frac{e+c-n}{2}} Q^e T^{-e} \paren{\sum_{k, i, j} \dim_\kb \overline{\hhh}(\beta)_{A=k, T=i, Q=j} A^kT^i Q^j} \]
    $\mathscr{P}_L$ is a link invariant and $\mathscr{P}_L|_{T=-1, A=-a^{2}q^2, Q=q}=\textnormal{HOMFLY}(L)$ up to a unit in $\Zb[q^{\pm1}, a^{\pm 1}]$.
\end{definition}

Combining this with our main results, we obtain the following necessary conditions for a link to be the closure of a positive braid.

\begin{theorem}
\label{posbraidcrit}
   Given $L=\widehat{\beta}$ for $\beta\in B_n$, let $M$ be the highest power of $T$ appearing in $\mathscr{P}_L$. Suppose $L$ is a positive braid link, e.g. $L=\widehat{\beta^\prime}$ where $\beta^\prime\in B_r^+$ for some $r$. Then
   \begin{enumerate}[(1)]
       \item If $L$ is not a disjoint union, then $\mathscr{P}_L$ is necessarily of the form
        \[  T^M(A^MQ^{-4M}) +O(T^{\le M-2}) \]
        \item If $L$ is prime then $\mathscr{P}_L$ is necessarily of the form
        \[  T^M(A^MQ^{-4M})\paren{1+T^{-2}Q^4+T^{-2}A}+O(T^{\le M-3}) \]
   \end{enumerate}
\end{theorem}
\begin{proof}
    $(1)$ If $L$ is not a disjoint union then it has at least one $\sigma_i$ for all $1\le i \le n-1$. Thus \cref{mainpostheorem} applies to $\beta^\prime$ and noting $e=|\beta^\prime|$ we obtain
    \[ \mathscr{P}_L= (Q^{-4}AT)^{\frac{|\beta^\prime|+c-r}{2}}Q^{|\beta^\prime|} T^{-|\beta^\prime|}\paren{  T^{|\beta^\prime|}Q^{-|\beta^\prime|}+ O(T^{\le |\beta^\prime|-2 }) }\]
    $(2)$ If $L$ is prime then \cref{primeiff} says there's a positive braid representative of $L$ satisfying the conditions of \cref{mainpostheorem2} and after normalizing as in the proof of (1), we obtain the result. Note such a braid representative also exists for $T(2,k)$ as $\sigma_1^k\sim \sigma_1^k \sigma_2\sim \sigma_1^{k-1}\sigma_2 \sigma_1 \sim \sigma_1^{k-2}\sigma_2 \sigma_1 \sigma_2$ .
\end{proof}

One should compare the result above with the computations of $\overline{\hhh}$ done in \cite[Appendix A]{EH19} for low crossing torus knots. For each computation, after factoring out the largest power of $t$, the two highest powers of $t$ are $1+qt^{-1}+at^{-1}=1+T^{-2}Q^4+T^{-2}A$. From \cref{t2ksect} we also see that 
\[T(AQ^{-4})(1+T^{-2}Q^4+T^{-2}A)=\mathscr{P}_{T(2,3)} \]
and so the theorem above says that after a grading shift, we have always have a copy of $\mathscr{P}_{T(2,3)}$ in $\mathscr{P}_{L}$ for any nontrivial positive prime braid link $L$. This leads to two natural questions.

\begin{question}
\label{ques2}
    Is there a spectral sequence from $\mathscr{P}_L$ to $\mathscr{P}_{T(2,3)}$ for any nontrivial positive braid link $L$?
\end{question}

Spectral sequences in link homologies are at the heart of many deep results in knot theory. For instance, Piccirillo's recent resolution of a long-standing unsolved problem in knot theory \cite{Pic20} relies on a link invariant constructed by Rasmussen \cite{Ras10} via a spectral sequence to ``homology" of the unknot. For such a spectral sequence to exist, 
homology of the unknot cannot be larger than homology of any knot. As demonstrated above, $\overline{\hhh}(T(2,3))$ is smaller than $\overline{\hhh}(\beta)$ for $\beta$ a nontrivial positive braid allowing for the possibility of \cref{ques2}.\\

For negative braids on 3 strands, an analogous result to \cref{posbraidcrit} holds using \cref{mainnegtheorem} and \cref{theorem:hhhneg2}. However, no such extension exists for braids on $\ge 4$ strands, since \cref{mainnegtheorem} and \cref{mainnegtheorem2} imply vanishing in the first three $T$-degrees. However for knots $K$ there is a workaround. Using \cref{lem:dualcomplex} and a bit of commutative algebra one obtains
\[ \mathscr{P}_K(A, T,Q)=\mathscr{P}_{K^\vee}(A^{-1}, T^{-1},Q^{-1}) \]
where $K^\vee$ is the mirror of $K$. Thus we have

\begin{theorem}
   Given $L=\widehat{\beta}$ for $\beta\in B_n$, let $-M$ be the lowest power of $T$ appearing in $\mathscr{P}_L$. If $L$ is a prime negative braid knot, then $\mathscr{P}_L$ is necessarily of the form
        \[  T^{-M}(A^{-M}Q^{4M})\paren{1+T^{2}Q^{-4}+T^{2}A^{-1}}+O(T^{\ge -M+3}) \]
\end{theorem}

\subsection{Outline}

\begin{itemize}
    \item In \cref{prelimsect} we first review some tools needed for our computations. \cref{clossect} establishes some results on the relationship between connect sums and the corresponding braid expression. \cref{t2ksect} contains the computation of the $R-$module structure on $\overline{\hhh}(T(2, \pm k))$.
    \item \cref{betabeta1sect} contains the computation of $\overline{\hhh}(\beta)$ in $T-$degrees $|\beta|$ and $|\beta|-1$. Here we make the key observation that when $\beta$ is a positive braid, the last 3 terms of ``$\hh^\bullet(F_\beta^\bullet)$" is a direct sum of multiple subcomplexes each of which is isomorphic to a Koszul complex, enabling quick computation of $\overline{\hhh}(\beta)$ in the last 2 $T-$degrees. 
    \item \cref{3strandsect} contains the computation of $\overline{\hhh}(\beta)$ in $T-$degree $|\beta|-2$ for $\beta\in B_3^+$. ``$\hh^\bullet(F_\beta^\bullet)$" stops looking Koszul at this point so we have to explicitly calculate the kernels and images. \cref{gensect} contains the analogous computation for $\beta\in B_n^+$ where we essentially reduce to the 3 strand case. 
    \item \cref{basisappend} uses \cite{Li22} to give an explicit diagrammatic basis for the Hochschild cohomology of Soergel Bimodules used in the computations in previous sections. We give a detailed proof of how to derive the basis for $\hh^\bullet(B_s B_t B_s)$ but this can be taken as a black box. Diagrammatics are simple and incredibly efficient in encapsulating information (especially since the underlying algebraic objects here are chain maps modulo chain homotopies) and so \cref{basisappend} will likely be of independent interest and utility to other researchers in the area. 
\end{itemize}

\subsection{Acknowledgments}

The author thanks Ben Elias for having the courage to apply a Reidemeister move. We would also like to thank Eugene Gorsky for answering his questions and Mikhail Khovanov and Chun-Ju Lai for feedback on an early draft of the paper. Finally, we express our gratitude to the Institute of Mathematics, Academia Sinica, for their exceptional research environment where this work was done.

\section{Preliminaries}
\label{prelimsect}
We urge the reader to read \cref{basisappend} first as the diagrammatic basis for the Hochschild cohomologies of Soergel Bimodules are the backbone of the calculations in this paper. 

\subsection{Notation}

\begin{itemize}
    \item Let $\sigma_i$ be the $i-$th positive crossing between the $i-$th and $(i+1)-$st strand and let $\sigma_i^{-1}$ be the $i-$th negative crossing as seen below \vspace{-1.5ex}
\begin{center}
    \raisebox{4ex}{$\sigma_i= \ $}
\begin{tikzpicture}[scale=0.6]
\braid[number of strands=4, gap=0.2] (braid) a_2^{-1} ;
\node at (1.5,-0.9) {$\cdots$};
\node at (3.5,-0.9) {$\cdots$};
\node at (2,-1.9) {$i$};
\node at (3.2,-1.9) {$i+1$};
\end{tikzpicture} \qquad \quad  \raisebox{4ex}{$\sigma_i^{-1}= \ $}
\begin{tikzpicture}[scale=0.6]
\braid[number of strands=4, gap=0.2] (braid) a_2 ;
\node at (1.5,-0.9) {$\cdots$};
\node at (3.5,-0.9) {$\cdots$};
\node at (2,-1.9) {$i$};
\node at (3.2,-1.9) {$i+1$};
\end{tikzpicture} \vspace{-2ex}
\end{center}
   \item Composition for us will always read bottom to top as we read right to left. For example, $\sigma_2 \sigma_1$ will be \vspace{-1ex}
\begin{center}
\begin{tikzpicture}[scale=0.7]
    \braid[number of strands=3, gap=0.2] (braid)  a_2^{-1}a_1^{-1}  ;    
\end{tikzpicture}
\end{center}
   \item Markov Move I will be cyclic permutation of the braid word, ie. $\sigma_1\sigma_2\sim \sigma_2 \sigma_1$ while Markov Move II will be de-stabilization, i.e. $\sigma_1^2\sigma_2\sim \sigma_1^2$.
   \item $(1)$ will shift the $Q-$grading down 1: $M(1)_i=M_{i+1}$.  
   \item $L_1\# L_2$ denotes the connect sum of $L_1$ and $L_2$.
   \item The positive braid index $b_p(L)$ of a link $L$ is the minimal $n$ s.t. $\exists \beta\in B_n^+$ s.t. $\widehat{\beta}=L$. Note, this may be larger than the regular braid index for $L$.
   \item By subexpression, we mean \un{non-contiguous} subexpression. Specifically, a subexpression of $a_1\ldots a_k$ is a word of the form $a_1^{\epsilon_1}\ldots a_k^{\epsilon_k}$ where $\epsilon_i\in \set{0,1}$. For example, $ss$ is a subexpression of $stts$.
\end{itemize}

\subsection{Diagrammatics}

We forgo a complete review of Soergel Calculus\cite{SC} and Ext Soergel Calculus\cite{Li22}, presenting only the generators and relations essential to this paper. For every ``color" or simple transposition $s\in S=\set{\sigma_1, \ldots, \sigma_{n-1}}$ we have

\begin{center}
        \begin{tabular}{c|c|c|c|c|c}
            generator &  \runit  & \raisebox{1ex}{\rcounit} & \rcup&  \rhzero &  \rhunit \ := 
\begin{tikzpicture}[xscale=0.3,yscale=0.5,thick,baseline]
 \draw[red] (0,-0.5) -- (0,0.75); \node[rhdot] at (0,0.2) {}; \node[rdot] at (0,-0.5) {};
\end{tikzpicture} \\[1ex]
            name & Startdot & Enddot & Cup &  Hochschild dot &  Hochschild Startdot \\[1ex]
            $(A,Q)$ degree& $(0,1)$  & $(0,1)$ & $(0, 0)$ & $(1,-4)$ & $(1, -3)$
        \end{tabular}
\end{center}

We also need the following generators. Let $R=\kb[\alpha_1, \ldots, \alpha_{n-1}]$ and $\Lambda=\Lambda^\bullet[\alpha_1^\vee, \ldots, \alpha_{n-1}^\vee]$

\begin{center}
        \begin{tabular}{c|c|c}
            generator &  $\boxed{f} \quad \forall f\in R$ homogeneous  & $\dboxed{\xi } \quad \forall \xi\in \Lambda$ homogeneous   \\[1ex]
            name & Box & Exterior Box    \\[1ex]
            bidegree& $(0,\deg f)$  & $(|\xi|,-2|\xi|)$ 
        \end{tabular}
    \end{center}

If one takes \cref{basisappend} as a black box (in particular \cref{hhsts}), then the computations in this paper should\footnote{You might need more, please see loc. cit. for additional relations if necessary.} only require the following relations

\begin{itemize}
    \item \textbf{Barbell} and \textbf{Hochschild Barbell}
    \[ \begin{tikzpicture}[xscale=0.3,yscale=0.5,thick,baseline]
 \draw[red] (0,-0.5) -- (0,0.75);  \node[rdot] at (0,-0.5) {}; \node[rdot] at (0,0.75) {};
\end{tikzpicture}= \boxed{\alpha_s} \qquad \qquad \qquad \begin{tikzpicture}[xscale=0.3,yscale=0.5,thick,baseline]
 \draw[red] (0,-0.5) -- (0,0.75);  \node[rhdot] at (0,-0.5) {}; \node[rdot] at (0,0.75) {};
\end{tikzpicture}= \dboxed{\alpha_s^\vee}   \]
  \item \textbf{Polynomial Forcing} and \textbf{Exterior Polynomial Forcing}
\[ \begin{array}{c}\begin{tikzpicture}[scale=0.5,thick,baseline]
 \draw[red] (0,-1.5) -- (0,1.5);
 \draw (0.5,-0.5) rectangle (1.5,0.5); \node at (1,0) {$f$};
\end{tikzpicture}\end{array}
=
\begin{array}{c}\begin{tikzpicture}[scale=0.5,thick,baseline]
 \draw[red] (0,-1.5) -- (0,1.5);
 \draw (-1.9,-0.5) rectangle (-0.5,0.5); \node at (-1.2,0) {$s(f)$};
\end{tikzpicture}\end{array}
+
\begin{array}{c}\begin{tikzpicture}[scale=0.5,thick,baseline]
 \draw[red] (0,0.8) -- (0,1.5); \node [rdot] at (0,0.8) {};
 \draw(-1,-0.5) rectangle (1,0.5); \node at (0,0) {$\partial_s(f)$};
 \draw[red] (0,-1.5) -- (0,-0.8); \node[rdot] at (0,-0.8) {};
\end{tikzpicture}\end{array} \qquad \qquad \qquad \begin{array}{c}\begin{tikzpicture}[scale=0.5,thick,baseline]
 \draw[red] (0,-1.5) -- (0,1.5);
 \draw[dashed] (0.5,-0.5) rectangle (1.5,0.5); \node at (1,0) {$\xi$};
\end{tikzpicture}\end{array}
=
\begin{array}{c}\begin{tikzpicture}[scale=0.5,thick,baseline]
 \draw[red] (0,-1.5) -- (0,1.5);
 \draw[dashed] (-1.9,-0.5) rectangle (-0.5,0.5); \node at (-1.2,0) {$s(\xi)$};
\end{tikzpicture}\end{array}
+
\begin{array}{c}\begin{tikzpicture}[scale=0.5,thick,baseline]
 \draw[red] (0,0.8) -- (0,1.5); \node [rhdot] at (0,0.8) {};
 \draw[dashed] (-1,-0.5) rectangle (1,0.5); \node at (0,0) {$\partial_s(\xi)$};
 \draw[red] (0,-1.5) -- (0,-0.8); \node[rdot] at (0,-0.8) {};
\end{tikzpicture}\end{array}\]
\item \textbf{1-color cohomology}
\[ \begin{tikzpicture}[xscale=0.3,yscale=0.5,thick,baseline]
 \draw[red] (0,-0.5) -- (0,0.75);  \node[rdot] at (0,-0.5) {}; 
\end{tikzpicture} \  \scalemath{0.9}{\dboxed{\alpha_s^\vee}}=\begin{tikzpicture}[xscale=0.3,yscale=0.5,thick,baseline]
 \draw[red] (0,-0.5) -- (0,0.75);  \node[rhdot] at (0,-0.5) {}; 
\end{tikzpicture} \ \boxed{\alpha_s}  \]
\end{itemize}

\subsection{Review of \cite{Mal24}}

Given a \un{positive} braid word $\beta\in \mathrm{Br}_{n}$, let $F_\beta^\bullet$ be the associated Rouquier complex and again let $|\beta|$ be the number of terms in the expression $\beta$. Then terms in $F_\beta^\bullet$ are then labeled by the $2^{|\beta|}$ possible subexpresssions of $\beta$. In \cite{Mal24}, they show that $F_\beta^\bullet$ is homotopy equivalent to the smaller complex $R_\beta^\bullet$ where $R_\beta^\bullet$ is the complex  
\[  R_\beta^q:= \bigoplus_{ \substack{  \un{x}\subseteq \beta  \\ |\beta|-|\un{x}|=q}} C_{\un{x}}(q)  \]
where $\un{x}\subseteq \beta$ means that $\un{x}$ is a subword (meaning distinct subexpression) of $\beta$ rather than a subexpression. For example, there is only one subword of length 1 of $ss$ while there are 2 possible subexpressions of $ss$ of length 1. $C_{\underline{x}}=\bs(\un{x^*})(|\un{x^*}|-|\un{x}|)$ where $\un{x^*}$ is obtained from $\un{x}$ by contracting each monotonous subsequence to a single letter. For example if $\beta=stst$ then there are 4 subwords $\un{x}$ of length 2 with corresponding $C_{\underline{x}}$
\[ ss: C_{ss}=B_s(-1) \qquad  st: C_{st}=B_sB_t \qquad ts: C_{ts}=B_tB_s \qquad tt: C_{tt}=B_t(-1)  \]
The differentials $d_{\un{x}}^{\un{z}}: C_{\un{x}}(|\beta|-|\un{x}|)\to C_{\un{z}}(|\beta|-|\un{z}|)$ in $R_\beta^\bullet$ are given as follows. Let $\un{x},\un{z}$ be subwords of $\beta$ and $d_s=\frac{1}{2}(\alpha_s\otimes_s 1 +1\otimes_s \alpha_s)$. Define
\[f_{s,k}=\begin{cases}
    \raisebox{1ex}{\rcounit} & \textnormal{ if }k=0 \\
    \ds & \textnormal{ if }k>0 \textnormal{ and odd} \\
    \rone & \textnormal{ if }k>0 \textnormal{ and even}
\end{cases}\]
Then $ d_{\un{x}}^{\un{z}}\neq 0 \iff \un{z}$ can be obtained from $\un{x}$ by deleting one letter in which case  $\un{x}$ and $\un{z}$ must be of the form
\[ \un{x}=\un{w_1} \underbrace{ss\ldots s}_{k+1} \un{w_2}, \qquad \un{z}=\un{w_1} \underbrace{s\ldots s}_{k} \un{w_2} \]
where $\un{w_1}$ does not end with $s$ and $\un{w_2}$ does not start with $s$. If $k>0$ or $\un{w_1}$ doesn't end with the same letter that $\un{w_2}$ starts with then the differential will be
\[ d_{\un{x}}^{\un{z}}=(-1)^{|\un{w_1}|} \mathrm{id}_{C_{\un{w}}}\otimes f_{s,k} \otimes \mathrm{id}_{C_{\un{w_2}}}\]
If $k=0$ and $\un{w_1}$ ends with the same letter that $\un{w_2}$ starts with then compose the RHS above with the corresponding trivalent vertex. For example, if $\un{x}=\ldots tst\ldots$ and $\un{z}=\ldots tt\ldots$, the differential will look like
\igc{0.45}{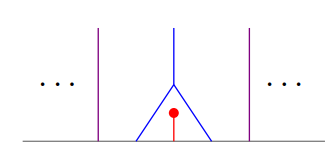}

For example, if $\beta=sstt$ where $s=\sigma_1, t=\sigma_t$, then $R_\beta^\bullet$ will be ($\un{B_sB_t(-2)}$ is in cohomological degree 0)

\[ \underline{B_sB_t(-2)} \xrightarrow{ \begin{pmatrix}
\rzero \ \dt \\
 \ds \ \bzero
\end{pmatrix} } B_s B_t\oplus B_s B_t \xrightarrow{\begin{pmatrix}
 \rzero \ \bcounit &  0 \\
 \ds \bzero &   -\rzero \dt     \\[1em]
 0       &  \rcounit \bzero        
\end{pmatrix}} B_s(1)\oplus B_sB_t(2)\oplus B_t(1) \xrightarrow{ \begin{pmatrix}
\ds & - \rzero \bcounit & 0 \\
 0  &  \rcounit \bzero  & \dt
\end{pmatrix} } B_s(3)\oplus B_t(3) \xrightarrow{\begin{pmatrix}
 \rcounit &\bcounit 
\end{pmatrix}} R(4) \]

\begin{remark}
    The results in this paper do not strictly need \cite{Mal24}. Summing over subwords instead of subexpressions gets rid of repeated $B_{\un{x}}$. Having repeated $B_{\un{x}}$ clearly doesn't change the images of the differentials and we can reduce the kernels of the differentials in the complex $\hh^k(F_\beta^\bullet)$ to those in $\hh^k(R_\beta^\bullet)$  by using the ``easy kernel" change of basis trick in \cref{fjker} and diagrammatics. However, the framework of \cite{Mal24} streamlines our exposition, and so we adopt it here for clarity.
\end{remark}

\subsection{Closure Results}
\label{clossect}
\begin{lemma}
\label{lem:csum}
    For $\beta\in B_n $, $\widehat{\beta}$ is a connect sum $\iff \exists \beta^\prime\in B_n$ s.t. $\widehat{\beta}=\widehat{\beta^\prime}$ and $\beta^\prime\in B_{k, k+1}$ for any $k$ where $B_{k, k+1}$ is the subset of $B_n$ consisting of words of the form $\beta_1\beta_2$ where $\beta_1\in \abrac{\sigma_1^{\pm 1}, \ldots, \sigma_{k-1}^{\pm 1}}$ and $\beta_2\in \abrac{\sigma_{k}^{\pm 1}, \ldots, \sigma_{n-1}^{\pm 1}}$.
\end{lemma}
\begin{proof}
    $\impliedby$ is clear while $\implies$ is a consequence of the composite braid theorem \cite{BM90}.
\end{proof}

\begin{lemma}
\label{1lowerlem}
Suppose $\beta$ is a braid on $n$ strands with only one $\sigma_i/\sigma_i^{-1}$. If $i=1, n-1$, then $\widehat{\beta}$ is the closure of a braid on $n-1$ strands. If $\ 2\le i \le n-2$, then $\widehat{\beta}$ is the connect sum of two braids on a smaller number of strands.
\end{lemma}
\begin{proof}
The first assertion is just Markov move II. For the second, can be seen visually as the single $\sigma_i$ provides the bridge that makes the connect sum.
\end{proof}

Thus for positive $\widehat{\beta}$ to not be the closure/connect sum of a braids(s) on lower strands, we must have that $\beta$ contains one of $sstt, stst, stts, ttss, tsts, tsst$ for $s\neq t\in \set{\sigma_1, \ldots, \sigma_{n-1}}$ as subexpressions. 

\begin{lemma}
\label{2lowerlem}
    $\beta\in B_n^+$ must contain $\sigma_i \sigma_{i+1} \sigma_i \sigma_{i+1}$ or $\sigma_{i+1} \sigma_i \sigma_{i+1} \sigma_i $ as a subexpression $\forall 1\le i\le n-2$ for $\widehat{\beta}$ to not be the closure/connect sum of a braids(s) on lower strands.
\end{lemma}
\begin{proof}
    First assume that $\beta$ only contains the subexpression $\sigma_i \sigma_i \sigma_{i+1} \sigma_{i+1}$ out of the 6 possibilities above. $\beta$ is then of the form
    \[\underbrace{\ldots}_{A}  \sigma_i\underbrace{\ldots}_{B}  \sigma_i\underbrace{\ldots}_{C} \sigma_{i+1} \underbrace{\ldots}_{D} \sigma_{i+1}\underbrace{\ldots}_{E}   \]
   where regions $A,B$ have no $\sigma_{i+1}$'s, regions $D,E$ have no $\sigma_{i}$'s by the assumption. As a result, any $\sigma_k$ for $k\in \set{1, \ldots, i-1}$ in regions $D,E$ can move to region $C$ starting from the left and similarly any $\sigma_\ell $ for $\ell\in \set{i+1, \ldots, n}$ in regions $A,B$ can move to region $C$ starting from the right. Now, the second $\sigma_i$ shown above can be assumed to be the last $\sigma_i$ while the first $\sigma_{i+1}$ can be assumed to be the first $\sigma_{i+1}$ so that region $C$ contains neither $\sigma_i$ or $\sigma_{i+1}$. As a result, we are free to move all the $\sigma_k$'s to the left of the $\sigma_\ell$'s. The resulting expression is in $B_{i, i+1}$ so by \cref{lem:csum} $\beta$ is a connect sum.\\
   
    Let $s=\sigma_i, t=\sigma_{i+1}$. The same applies for $sstt$. Using Markov move I, one can turn $\beta$ with subexpression $stts$ or $tsst$ to one with only $sstt$ and thus only $stst$ or $tsts$ remain to not be a connect sum of braids on a smaller number of strands. 
\end{proof} 

\begin{prop}
\label{primeiff}
     Let $\beta\in B_n^+$ where $n$ is the positive braid index for $\widehat{\beta}$. Then $\widehat{\beta}$ is not the connect sum of braids(s) on lower strands $\iff$ $\beta$ contains $\sigma_i \sigma_{i+1} \sigma_i \sigma_{i+1}$ or $\sigma_{i+1} \sigma_i \sigma_{i+1} \sigma_i $ as a subexpression $\forall 1\le i\le n-2$.
\end{prop}
\begin{proof}
    The lemma above proves the $\implies$ direction. For the $\impliedby$ direction by \cref{lem:csum} we need to show for all $\beta^\prime$ s.t. $\widehat{\beta^\prime}=\widehat{\beta}$, $\beta^\prime\not\in B_{k,k+1}$ for any $k\ge 2$. let $s=\sigma_{k-1}, t=\sigma_k$. It is clear that as long as condition
    $$ (k) \ \beta^\prime \textnormal{ has }stst \textnormal{ or }tsts \textnormal{ as a subexpression}$$
    holds, then $\beta^\prime\not\in B_{k,k+1}$. Markov move I preserves this property and so we only have to consider $\beta^\prime$ obtained from $\beta$ by braid relations. $\beta$ starts out satisfying condition $(k)$ for all $k$ and thus is never in $B_{k, k+1}$ for any $k$. If there is a sequence of braid moves that transforms $\beta$ into a word that doesn't satisfy $(k_0)$ for some $k_0$, the final move must be either move $(A)$ or $(B)$ below.
 \[  (A) \ \ldots sts\ldots\to \ldots tst\ldots \qquad  (B) \ \ldots tst \ldots \to \ldots sts \ldots  \]
where now $s=\sigma_{k_0-1}, t=\sigma_{k_0}$, as this is the only way to change the relative position of $s$ and $t$. WLOG, we can assume that before applying $(A)$ or $(B)$ the word satisfies $(k)$ for all $(k)$. If $(A)$ was the final move, then the $\ldots$ on either side cannot contain any $s$. But $(A)$ doesn't change the relative position of $s$ and $\sigma_{k_0-2}$ so that means before applying $(A)$ condition $(k_0-1)$ was not satisfied, a contradiction. The exception is $k_0=2$, but now we can remove the only $\sigma_1$ using Markov move II after move $(A)$ contradicting minimality of $\beta$. Similarly, if $(B)$ was the final move, then the $\ldots$ on either side cannot contain any $t$ and a similar argument holds with $t$ and $\sigma_{k_0+1}$ where the exception is $k_0=n-1$ but in which case we can remove the only $\sigma_{n-1}$ using Markov move II after move $(B)$, again contradicting minimality of $\beta$.

\end{proof}

\begin{remark}
 The analogous results clearly also holds for negative links.  
\end{remark}

For connect sums, their $\overline{\hhh}$ can be computed using the following 

\begin{prop}[{\cite[Lemma 7.8]{Ras15}}]
\label{raslemma}
Given two braids $\beta_1$ and $\beta_2$, as graded vector spaces, we have that 
\[ \overline{\hhh}(\beta_1\#\beta_2)= \overline{\hhh}(\beta_1)\otimes_{\kb}  \overline{\hhh}(\beta_2)\]
\end{prop}

\subsection{$\overline{\hhh}$ for $T(2,\pm k)$}
\label{t2ksect}
The reduction results in the previous subsection need to start with a base case, namely $n=2$. There are only two possible families of links that appear as the closure of $\beta\in B_2$, namely $T(2,m)=\widehat{\sigma_1^m}$ and $T(2,-m)=\widehat{\sigma_1^{-m}}$ where $m\in \mathbb{Z}^+$. We record the $R=\kb[\alpha_s] $ module structure on their $\overline{\hhh}$ below. 

\begin{center}
    $\paren{\overline{\hhh}(\sigma_1^{m})^{A}}^{T=m-i}=$\begin{tabular}{c|c}
  $A=0$  &  $A=1$  \\\hline
$\begin{cases}
 \kb(m) & \textnormal{ if } i=0 \\
 \kb(m-2i)  &   \textnormal{ if }2\le i \le m, i \textnormal{ even} \\
\kb[\alpha_s](-m)& \textnormal{ if }i=m , m \textnormal{ even} \\
 0  & \textnormal{ otherwise}
\end{cases}$ & $\begin{cases}
 0& \textnormal{ if } i=0 \\
 \kb(m-2i+4)  &   \textnormal{ if }2\le i \le m, i \textnormal{ even} \\
\kb[\alpha_s](-m+4)& \textnormal{ if }i=m , m \textnormal{ even} \\
 0  & \textnormal{ otherwise}
\end{cases}$
\end{tabular}
\end{center}

\begin{center}
    $\paren{\overline{\hhh}(\sigma_1^{-m})^{A}}^{T=-m+i}=$\begin{tabular}{c|c}
  $A=0$  &  $A=1$  \\\hline
$\begin{cases}
 0 & \textnormal{ if } 0\le i <3 \\
 \kb (-m+2i-2) &   \textnormal{ if }3\le i \le m, i \textnormal{ odd} \\
  \kb[\alpha_s](m-2) & \textnormal{ if }i=m , m \textnormal{ even} \\
  0 & \textnormal{ otherwise}
\end{cases}$ & $ \begin{cases}
 0 & \textnormal{ if } i=0 \\
 \kb (-m+2i+2) &   \textnormal{ if }1\le i \le m \textnormal{ and } i \textnormal{ odd} \\
  \kb[\alpha_s](m+2) & \textnormal{ if }i=m \textnormal{ and } m \textnormal{ even} \\
  0 & \textnormal{ otherwise}
\end{cases}$
\end{tabular}
\end{center}

Note that $\overline{\hhh}(\sigma_1^{m})$ is supported entirely in even or odd $T-$degrees while this is only true when $m$ is odd for $\overline{\hhh}(\sigma_1^{-m})$. Also note that $\overline{\hhh}(\sigma_1^{\pm m})$ is f.d. $\iff m$ is odd. 

\subsection{Koszul Complexes}

Let $R$ be a commutative ring and $E=R^r$ , and let $\phi: E\to R$ be a $R-$module homomorphism. In particular, if $\set{e_i}_{i=1}^r$ is a $R-$basis for $E$ then $\phi$ is the same as a tuple $(s_1,\ldots, s_r)$ of elements in $R$ where $\phi(e_i)=s_i$. The Koszul complex associated to $\phi$ will be the following complex of $R-$modules.
\begin{equation}
\label{koszuleq}
K_\bullet(\phi)=K_\bullet(s_1,\ldots, s_r): 0 \to \wedge^rE\to \ldots \wedge^{2} E \xrightarrow{d_2} \wedge^1 E \xrightarrow{d_1} \un{R} \to 0      
\end{equation}

where $R$ is in \un{homological} degree 0 and the $R-$linear differential is given by "contraction"
\[ d_k(e_1\wedge \cdots \wedge e_k)=\sum_{i=1}^k (-1)^{i+1}\phi (e_i) e_1 \wedge \cdots \wedge \widehat{e_i} \wedge \cdots \wedge e_k \]

\begin{lemma}
    $K_\bullet(s_1,\ldots, s_r)=$cone of multiplication by $s_r$ map between the complexes: $K_\bullet(s_1,\ldots, s_{r-1})\to K_\bullet(s_1,\ldots, s_{r-1})$.
\end{lemma}

\begin{corollary}
\label{koszulvanish}
If $s_r=u$ is a unit in $R$, then $K_\bullet(s_1,\ldots, u)$ is exact everywhere.
\end{corollary}
\begin{proof}
    Because $u$ is a unit, the map $K_\bullet(s_1,\ldots, s_{r-1})\xrightarrow{u} K_\bullet(s_1,\ldots, s_{r-1})$ is an isomorphism and in particular a quasi-isomorphism and thus the cone is acylic as desired.
\end{proof}

\section{$T-$degrees $|\beta|$ and $|\beta|-1$}
\label{betabeta1sect}

\fbox{Let $\beta\in \mathrm{Br}_n$ be a positive braid with at least one $\sigma_i$ for $1\le i \le n-1$.} \cref{1lowerlem} and \cref{raslemma} implies that computing $\overline{\hhh}$ for this class of braids computes it for all positive braids. For $n$ strands, $R=\kb[\alpha_1, \ldots, \alpha_{n-1}]$. In this section we are only computing $T-$degrees $|\beta|$ and $|\beta|-1$ and so the last three terms of $R^\bullet_\beta$ can be written as

\begin{equation}
\label{rbetaeq}
    \to \overbrace{B_1B_2(|\beta|-2)\oplus\cdots \oplus B_{n-2}B_{n-1}(|\beta|-2)}^{R^{|\beta|-2}_\beta}\underset{d_{|\beta|-2}^\prime}{\xrightarrow{ 
   \lcounit{i} \lzero{j}  -\lzero{i} \lcounit{j} 
   }} \overbrace{B_1(|\beta|-1)\oplus \cdots \oplus B_{n-1}(|\beta|-1)}^{R^{|\beta|-1}_\beta} \underset{d_{|\beta|-1}}{ \xrightarrow{\begin{pmatrix}
 \lcounit{1}& \ldots &\lcounit{n} 
\end{pmatrix}}} R(|\beta|) 
\end{equation}  
where $ R(|\beta|) $ is placed in cohomological degree $|\beta|$ and $d_{|\beta|-2}^\prime$ is of the form above because we have omitted the $C_{ii}$ terms (their image under $d_{|\beta|-2}$ is 0). Our assumption on $\beta$ only guarantees one of $C_{ij}$ or $C_{ji}$ appears in $R^\bullet_{\beta}$ but since $\hh^\bullet(d_{|\beta|-2}^\prime)(C_{ij})=\hh^\bullet(d_{|\beta|-2}^\prime)(C_{ji})$ (this is clear diagrammatically, one is the negative of the other, see \cref{basisappend} for what $\hh^k(B_iB_j)$ looks like) we will only include the $C_{ij}$ terms for $i<j$ in $R_\beta^{|\beta|-2}$. In fact, one may suggestively write $R^{|\beta|-2}_\beta``="\Lambda^2(B_1\oplus \ldots\oplus B_n)$

\begin{definition}
   For $J\subseteq [n-1]$, let $\alpha_J^\vee=\prod_{j\in J} \alpha_j^\vee$ and $\widehat{\alpha_J^i}^\vee=\prod_{i\neq j \in J} \alpha_j^\vee$.
\end{definition}

\begin{lemma}
\label{gjlem}
For $J\subseteq [n-1]$ and $|J|=k$ let 
 \[ G_{J}=\bigoplus_{i_0\in J}\lhunit{i_0} \scalemath{0.85}{\dboxed{\widehat{\alpha_J^{i_0}}^\vee}} \, R\oplus\bigoplus_{j\not\in J}\lunit{j}\scalemath{0.85}{\dboxed{\alpha_{J}^\vee}} \, R\]
 Then $\displaystyle \hh^k(R^{|\beta|-1}_\beta)=\bigoplus_{|J|=k} G_J$ and $\hh^k(d_{|\beta|-1})(G_J) \subseteq \scalemath{0.85}{\dboxed{\alpha_{J}^\vee}} \, R $ $\implies \hh^k(d_{|\beta|-1})$ is a block matrix.
\end{lemma}
\begin{proof}
    The first follows from \cref{basisappend} while the second from (Hochschild) barbell relations.
\end{proof}

\begin{lemma}
\label{fjgjlem}
 For $J\subseteq [n-1]$ and $|J|=k$ let
\[ F_{J}^\prime=\bigoplus_{\substack{i_0,j_0\in J \\ i_0<j_0}} \lhunit{i_0} \lhunit{j_0} \scalemath{0.85}{\dboxed{\widehat{\alpha_J^{i_0, j_0}}^\vee}}R\oplus\bigoplus_{\substack{i_0\in J,j\not\in J \\ j< i_0}} \lunit{j}\lhunit{i_0} \scalemath{0.85}{\dboxed{\widehat{\alpha_J^{i_0}}^\vee}} R\oplus\bigoplus_{\substack{i_0\in J, j\not\in J \\ i_0<j }}\lhunit{i_0}  \lunit{j} \scalemath{0.85}{\dboxed{\widehat{\alpha_J^{i_0}}^\vee}} R\oplus \bigoplus_{\substack{i,j\not\in J\\ i<j }} \lunit{i}\lunit{j}\scalemath{0.85}{\dboxed{\alpha_{J}^\vee}} R        \]
Then $\displaystyle \hh^k(R^{|\beta|-2}_\beta)=\bigoplus_{|J|=k} F_J^\prime $ and $\hh^k(d_{|\beta|-2}^\prime)\paren{F_{J}^\prime}\subseteq G_{J}$ $\implies \hh^k(d_{|\beta|-2}^\prime)$ is a block matrix.
\end{lemma}
\begin{proof}
The proof is analogous to above.
\end{proof}

The two lemmas above imply that the last three terms of $\hh^k(R_\beta^\bullet)$ actually splits as a direct sum of complexes
\begin{equation}
\label{rbetasplit}
    \ldots \bigoplus_{|J|=k} \paren{\to F_J^\prime \to G_J\to \scalemath{0.85}{\dboxed{\alpha_{J}^\vee}} \, R } 
\end{equation} 

\begin{example}
\label{ex1}
    Take $n=4$ so $R=\kb[\alpha_1, \alpha_2, \alpha_3]$. Take $J=\set{1,2}\subseteq \set{1,2,3}$. The last 3 terms of the $J$ part of the complex $\hh^2(R^\bullet_\beta)$ (modulo $Q$ grading shifts) will be
\[ \scalemath{0.9}{\lhunit{1}\lhunit{2} R \oplus \lhunit{1}\lunit{3} \scalemath{0.85}{\dboxed{\alpha_2^\vee}}  R \oplus \lhunit{2}\lunit{3} \scalemath{0.85}{\dboxed{\alpha_1^\vee}} R\ \raisebox{-3.5ex}{$\underset{\hh^2(d_{|\beta|-2}^\prime)}{\xrightarrow{\begin{pmatrix}
    -1  & -\alpha_3  & 0 \\
    1 & 0 & -\alpha_3 \\
     0  & 1  & 1
\end{pmatrix}}}$} \lhunit{1} \scalemath{0.85}{\dboxed{\alpha_2^\vee}} R\oplus \lhunit{2} \scalemath{0.85}{\dboxed{\alpha_1^\vee}} R \oplus \lunit{3} \scalemath{0.85}{\dboxed{\alpha_1^\vee\alpha_2^\vee}} R \ \raisebox{-2ex}{$\underset{\hh^2(d_{|\beta|-1})}{\xrightarrow{\begin{pmatrix}
    1 & 1 & \alpha_3
\end{pmatrix}}}$} \scalemath{0.85}{\dboxed{\alpha_1^\vee\alpha_2^\vee}} R } \]
  One then sees that this is the same as the last three terms of $K_\bullet(1,1,\alpha_3)$ and thus by \cref{koszulvanish} the cohomology at the last 2 spots is 0.   
\end{example}

\begin{maintheorem}
\label{mainpostheorem}
    Let $\beta$ be \un{any} positive braid on $n$ strands with at least one $\sigma_i$ for $1\le i \le n-1$. Then as $R-$modules
  \[ (a) \ \paren{\overline{\hhh}(\beta)^{A=0, T=\ell}}=  \begin{cases} \kb(|\beta|) & \ell=|\beta| \\
    0 & \ell=|\beta|-1
\end{cases} \qquad (b) \ \paren{\overline{\hhh}(\beta)^{A=k, T=\ell}}=  \begin{cases} 0 & \ell=|\beta| \\
    0 & \ell=|\beta|-1
\end{cases}  \quad \forall k\ge 1\]
\end{maintheorem}
\begin{proof}
    $(a)$ When $A=0$, there is only one summand in \cref{rbetasplit}, and it's isomorphic to the last 3 terms in $K_\bullet(\alpha_1, \ldots, \alpha_{n-1})$. The isomorphism is because the differential in \cref{rbetaeq} alternates signs and the $R-$basis for $\hh^0(B_iB_j)$ and $\hh^0(B_i)$ are all comprised of dot morphisms so the barbell relation contracts one component at a time. Since $\alpha_1, \ldots, \alpha_{n-1}$ is a regular sequence in $R=\kb[\alpha_1, \ldots, \alpha_{n-1}]$, the result follows.\\

    $(b)$ When $A=k\ge 1$, \cref{rbetasplit} has $\binom{n-1}{k}$ summands corresponding to the $\binom{n-1}{k}$ different subsets $J$ of $[n-1]$ of size $k$. The $J$ summand 
    will be isomorphic to $K_\bullet(\set{\alpha_i }_{i\not\in J}, \overbrace{ 1, \ldots, 1}^{|J| \ many})$ as once again the differential alternates signs and the $R-$basis for $\hh^k(B_iB_j)$ and $\hh^k(B_i)$ are all comprised of Hochschild dot and dot morphisms+exterior boxes. Thus the barbell and hochschild barbell relations will contract one component at a time. Since $k\ge 1$, by \cref{koszulvanish} it follows that the cohomology of each $J$ summand is 0 everywhere. 
\end{proof}

\subsection{Negative Braids}
Given a positive braid $\beta$ let $\beta^\vee$ be the mirror dual, i.e. switch all positive crossings to negative crossings. Applying \cite[Corollary 1.12]{GHM19} ($\sbim_n=\sbim(\mathbb{C}^n, S_n)$ in \cite{GHM19} while $\sbim_n=\sbim(\mathbb{C}^n/(x_1+\ldots+x_n) , S_n)$ in our paper and hence the grading shift is by $2(n-1)$ instead of $2n$ below.), 
\begin{lemma}
\label{lem:dualcomplex}
    Let $\beta$ be a braid on $n$ strands. For $F^\bullet_{\beta}\in K^b(\sbim_n)$ we have an isomorphism of complexes 
    \[\hh^k(F^\bullet_{\beta^\vee}) \cong \un{\hom}_R^\bullet(\hh^{n-1-k}(F^\bullet_{\beta}), R) (2(n-1)) \]
\end{lemma}

Thus for $\alpha$ a \un{negative} braid word on $n$ strands, it follows that the first 3 terms of the complex computing $\overline{\hhh}(\alpha)$ are obtained by applying $(-)^\vee=\hom_R(-, R)$\footnote{aka reversing arrows and taking the matrix transpose.} to the last 3 terms of the complexes $\hh^k(R_\beta^\bullet)$. $R^\bullet_\beta$ can still be written as in \cref{rbetaeq}, as multiples of columns now becomes multiples of rows so deleting them doesn't change the kernel of $\hh^k(d_{|\beta|-2})^T$. It is well known that the Koszul complex is self dual: the cohomology of $K_i(s_1, \ldots, s_r)^\vee$ is equal to the cohomology of $K_{r-i}(s_1, \ldots, s_r)$. Thus we see that in either case in \cref{mainpostheorem}, the cohomology is always 0 (as long as $n-1\ge 2$) and thus

\begin{maintheorem}
\label{mainnegtheorem}
    Let $\alpha$ be \un{any} negative braid on $n\ge 3$ strands with at least one $\sigma_i^{-1}$ for $1\le i \le n-1$. Then as $R-$modules
  \[  \paren{\overline{\hhh}(\alpha)^{A=k, T=\ell}}=  \begin{cases} 0 & \ell=-|\alpha| \\
    0 & \ell=-|\alpha|+1
\end{cases}  \quad \forall k\]
\end{maintheorem}

\section{$T-$degree $|\beta|-2$ for 3 strands}
\label{3strandsect}
Let $s=\sigma_1$ and $t=\sigma_2$ and $R=\kb[\alpha_s, \alpha_t]$.

\begin{theorem}
\label{theorem:hhhpos2}
Let $\beta$ be any positive braid on 3 strands containing $stst$ or $tsts$ as a subexpression. Then
\[ \paren{\overline{\hhh}(\beta)^{A=0}}^{T=|\beta|-2}=  \kb(|\beta|-4), \qquad \paren{\overline{\hhh}(\beta)^{A=1}}^{T=|\beta|-2}= \kb(|\beta|), \qquad \paren{\overline{\hhh}(\beta)^{A=2}}^{T=|\beta|-2}=  0   \]
\end{theorem}
\begin{proof}
First suppose that our positive braid $\beta\in \mathrm{Br}_{3} $ contains $ststst$ as a subexpression. Then $\beta$ contains all possible subexpressions of length $3$ and the order for the possible subexpressions of length 3 we will use is
\[ sss \qquad sst \qquad sts \qquad tss \qquad stt \qquad tst \qquad tts \qquad ttt \]
Then the penultimate three terms of $R^\bullet_{\beta}$ will be 
\begin{align*}
    &B_s(|\beta|-5)\oplus B_sB_t(|\beta|-4)\oplus B_sB_tB_s(|\beta|-3)\oplus B_tB_s(|\beta|-4) \oplus B_sB_t(|\beta|-4)\oplus B_tB_sB_t (|\beta|-3)\oplus B_tB_s(|\beta|-4)\oplus B_t(|\beta|-5) \\
    & \underset{d_{|\beta|-3}}{\xrightarrow{ \begin{pmatrix}
        \rone & \rzero \bcounit & -\rpitchfork & \bcounit \rzero & 0 & 0 & 0 & 0 \\
        0 & \ds \bzero & \rzero \ \bzero \ \rcounit & 0 & -\rzero \dt & \bcounit \ \rzero \ \bzero & 0 & 0 \\
        0 & 0 & \rcounit \ \bzero \ \rzero & - \bzero \ds & 0 & \bzero \ \rzero \ \bcounit & \dt \ \rzero & 0 \\
        0 & 0 & 0 & 0 & \rcounit \bzero & - \bpitchfork & \bzero \rcounit & \bone
    \end{pmatrix}
    }} B_s(|\beta|-3) \oplus B_sB_t (|\beta|-2) \oplus B_tB_s(|\beta|-2) \oplus B_t (|\beta|-3) \\
&\underset{d_{|\beta|-2}}{\xrightarrow{ \begin{pmatrix}
\ds & - \rzero \bcounit &  \bcounit \rzero & 0 \\
 0  &  \rcounit \bzero &  -\bzero \rcounit & \dt
\end{pmatrix} }}
     B_s(|\beta|-1)\oplus B_t(|\beta|-1)
\end{align*}

$\bullet$ The partial complex for $\overline{\hhh}^{A=0}$ will then be (for notation purposes, let $\overline{ \runit (n)}= \runit \, R(|\beta|+n)$, etc)

\begin{align*}
    &\overline{\runit(-5)}\oplus \overline{\runit\bunit(-4)} \oplus \overline{\runit\bunit \runit(-3)}\oplus \overline{\rpitchcup (-3)} \oplus \overline{\bunit \runit(-4)} \oplus \overline{\runit\bunit (-4)}\oplus \overline{\bunit \runit \bunit (-3)} \oplus \overline{\bpitchcup (-3)}\oplus \overline{\bunit \runit(-4)}\oplus \overline{\bunit (-5)}  \\
    & \scalemath{0.9}{\underset{\hh^0(d_{|\beta|-3})}{\xrightarrow{\begin{pmatrix}
        \alpha_s & \alpha_t & -\alpha_t & 1& \alpha_t & 0 & 0 & 0 & 0& 0 \\
        0 & 0 & \alpha_s & 1 & 0 & 0 & \alpha_t & 1 & 0  & 0\\
        0 & 0 & \alpha_s & 1 & 0 & 0 & \alpha_t & 1 & 0 & 0 \\
        0 & 0 & 0  & 0 & 0 & \alpha_s & -\alpha_s & 1 & \alpha_s & \alpha_t
    \end{pmatrix}}}\overline{\runit(-3)}\oplus \overline{\runit\bunit(-2)} \oplus \overline{\bunit \runit(-2)} \oplus \overline{\bunit(-3)} \underset{\hh^0(d_{|\beta|-2})}{\xrightarrow{ \begin{pmatrix}
0 & - \alpha_t &  \alpha_t & 0 \\
 0  &  \alpha_s &  -\alpha_s & 0
\end{pmatrix} }}\overline{\runit(-1)} \oplus \overline{\bunit(-1)}}
\end{align*}
Let $\vec{e_i}$ be the $i-$th elementary basis vector of the matrix for $\hh^0(d_{|\beta|-2})$. Then $\ker \hh^0(d_{|\beta|-2}) = R \vec{e_1}\oplus R (\vec{e_2}+\vec{e_3})  \oplus R \vec{e_4} $. Applying the invertible matrix
\[ \begin{bmatrix}
    1 & 1 & 0  \\
    0 & 1 & 1 \\
    0 & 0 & 1  
\end{bmatrix} \implies \ker \hh^0(d_{|\beta|-2})=  R \vec{e_1}\oplus R (\vec{e_1}+\vec{e_2}+\vec{e_3}) \oplus R (\vec{e_2}+\vec{e_3} +\vec{e_4})   \]
Let $\vec{w_i}$ be the $i-$th column of $\hh^0(d_{|\beta|-3})$. 
Since $\vec{e_1}=\vec{e_4}+\vec{w_3}-\vec{w_7}$, we see that $\vec{w_6}=\alpha_s \vec{e_1}+\im \hh^0(d_{|\beta|-3})$. Thus $\vec{w_6}, \vec{w_2}, \vec{w_4}, \vec{w_8}$ generate $\im \hh^0(d_{|\beta|-3})$ as right $R$ modules. Since $\vec{w_4}=\vec{e_1}+\vec{e_2}+\vec{e_3}, \vec{w_8}=\vec{e_2}+\vec{e_3} +\vec{e_4}$
\begin{align*} \paren{\overline{\hhh}^{A=0}}^{T=|\beta|-2}=\ker \hh^0(d_{|\beta|-2}) / \im \hh^0(d_{|\beta|-3})=\overline{\runit(-3)}/(\alpha_s, \alpha_t) =\kb(|\beta|-4) 
\end{align*}
$\bullet$ The partial complex for $\overline{\hhh}^{A=1}$ will be

\begin{align*}
    &\scalemath{0.8}{\overline{\rhunit(-5)}\oplus \overline{ \runit\ \raisebox{-1ex}{ \dbox{\scalemath{0.8}{\alpha_t^\vee}}}(-5)} \oplus \overline{\rhunit \bunit(-4)}\oplus \overline{\runit \bhunit(-4)} \oplus \overline{\rhunit \bunit \runit (-3)}  \oplus \overline{\rhpitchcupout (-3)}   \oplus \overline{\rhpitchcupin (-3)}  \oplus \overline{ \runit \bhunit \runit (-3)} \oplus \overline{\bhunit \runit(-4)}\oplus \overline{\bunit \rhunit(-4)} \oplus \overline{\rhunit \bunit(-4)}\oplus \overline{\runit \bhunit(-4)} \oplus \overline{\bhunit \runit \bunit (-3)}  \oplus \overline{\bhpitchcupout (-3)}   }  \\
    & \scalemath{0.8}{  \oplus \overline{\bhpitchcupin (-3)} \oplus \overline{ \bunit \rhunit \bunit (-3)}\oplus \overline{\bhunit \runit(-4)}\oplus \overline{\bunit \rhunit(-4)} \oplus \overline{\bhunit(-5)}\oplus \overline{ \bunit\ \raisebox{-1ex}{ \dbox{\scalemath{0.8}{\alpha_s^\vee}}}(-5)}}  \underset{\hh^1(d_{|\beta|-3})}{
    \xrightarrow{
    \scalemath{0.78}{\begin{pmatrix}
    \alpha_s & 0 & \alpha_t & 0 & -\alpha_t & 1 & 1 & 0 & 0 &\alpha_t & 0 & 0 & 0 & 0 & 0 & 0 & 0 &0 & 0 &0 \\
    0 & \alpha_s & 0 & 1 & 0 & 0 &0 & -1 & 1 & 0 & 0 & 0 & 0 & 0 &0 &0 & 0 & 0 & 0 &0 \\
    0 & 0 & 0 & 0 &\alpha_s & 1 & 0 & 0 & 0 & 0 & 0 & 0 & 0 & 0 & 1 & \alpha_t &0 &0 & 0 & 0 \\
    0 & 0 &0 &0 &0 &0 & 1 & \alpha_s &0 & 0 & 0 & 0 &\alpha_t & 1 & 0 & 0 &0 &0 &0 & 0 \\
    0 & 0 &0 &0 &0 &0 & 1 & \alpha_s & 0 & 0 &0 &0 &\alpha_t & 1 & 0 &0 &0 &0 &0 &0 \\
    0 & 0 &0 &0 &\alpha_s &1 & 0 & 0 & 0 & 0& 0 &0 &0 &0 & 1 & \alpha_t & 0 &0 &0 &0 \\
    0 & 0 &0 &0 &0 &0 & 0 & 0 & 0 & 0& 0 &\alpha_s & -\alpha_s &1 &1 & 0 & \alpha_s & 0 &\alpha_t &0 \\
    0 & 0 &0 &0 &0 &0 & 0 & 0 & 0 & 0& 1 & 0 & 0 & 0 & 0 & -1 & 0 & 1 & 0 &\alpha_t
    \end{pmatrix}
    }} 
    }\\
    &\overline{\rhunit(-3)}\oplus \overline{ \runit\ \raisebox{-1ex}{ \dbox{\scalemath{0.8}{\alpha_t^\vee}}}(-3)} \oplus \overline{\rhunit \bunit(-2)}\oplus \overline{\runit \bhunit(-2)} \oplus \overline{\bhunit \runit(-2)}\oplus \overline{\bunit \rhunit(-2)} \oplus \overline{\bhunit(-3)}\oplus \overline{ \bunit\ \raisebox{-1ex}{ \dbox{\scalemath{0.8}{\alpha_s^\vee}}}(-3)} \underset{\hh^1(d_{|\beta|-2})}{\xrightarrow{\begin{pmatrix}
        0 & 0 & -\alpha_t & 0 & 0 &\alpha_t & 0 &0 \\
        0 & 0 & 0 & -1 & 1 & 0 & 0 &0 \\
        0 & 0 &0 &\alpha_s & -\alpha_s & 0 & 0&0\\
        0 & 0 & 1 & 0 & 0 & -1 & 0 &0
    \end{pmatrix}}} \\
    &\overline{\rhunit(-1)}\oplus \overline{ \runit\ \raisebox{-1ex}{ \dbox{\scalemath{0.8}{\alpha_t^\vee}}}(-1)} \oplus \overline{\bhunit(-1)}\oplus \overline{ \bunit\ \raisebox{-1ex}{ \dbox{\scalemath{0.8}{\alpha_s^\vee}}}(-1)}
\end{align*}
As before, let $\vec{e_i}$ be the $i-$th elementary basis vector of the matrix for $\hh^1(d_{|\beta|-2})$. Then $\ker \hh^1(d_{|\beta|-2}) = R \vec{e_1}\oplus R (\vec{e_4}+\vec{e_5}) \oplus R (\vec{e_3}+\vec{e_6}) \oplus R \vec{e_7} \oplus R \vec{e_2} \oplus R \vec{e_8}$. Applying the invertible matrix
\[ \begin{bmatrix}
    1 & 1 & 0 & 0 & 0 & 0\\
    0 & 1 & 0 & 1 & 0 & 0\\
    0 & 0 & 1 & 0 & 0 & 0 \\
    0 & 0 & 1 & 1 & 0 & 0 \\
    0 & 0 & 0 & 0 & 1 & 0 \\
    0 & 0 & 0 & 0 & 0 & 1
\end{bmatrix} \implies \ker \hh^1(d_{|\beta|-2})=  R \vec{e_1}\oplus R (\vec{e_1}+\vec{e_4}+\vec{e_5}) \oplus R (\vec{e_3}+\vec{e_6} +\vec{e_7}) \oplus R(\vec{e_4}+\vec{e_5}+ \vec{e_7})  \oplus R \vec{e_2} \oplus R \vec{e_8} \]
As before, let $\vec{w_i}$ be the $i-$th column of $\hh^1(d_{|\beta|-3})$. Then $$\vec{w_7}=\vec{e_1}+\vec{e_4}+\vec{e_5}, \qquad \vec{w_{15}}=\vec{e_3}+\vec{e_6} +\vec{e_7}, \qquad \vec{w_{14}}=\vec{e_4}+\vec{e_5} +\vec{e_7}, \qquad \vec{w_{15}}= \vec{w_7}-\vec{w_6}-\vec{w_{14}} $$
 As a result, $\vec{w_3}, \vec{w_4},\vec{w_6}, \vec{w_7}, \vec{w_{11}}, \vec{w_{12}}, \vec{w_{14}}$ generate $\im \hh^1(d_{|\beta|-3})$. Moreover, $\vec{e_1}=\vec{e_7}+\vec{w_7}-\vec{w_{14}}$, so as right $R$ modules,
\begin{align*} \paren{\overline{\hhh}^{A=1}}^{T=|\beta|-2}&=\ker \hh^1(d_{|\beta|-2}) / \im \hh^1(d_{|\beta|-3})=\overline{\rhunit(-3)} /(\alpha_t, \alpha_s)=\kb(|\beta|)
\end{align*}
$\bullet$ The partial complex for $\overline{\hhh}^{A=2}$ will be
\begin{align*}
&\overline{\rhunit\ \raisebox{-1ex}{ \dbox{\scalemath{0.9}{\alpha_t^\vee}}} (-5)} \oplus \overline{\rhunit \bhunit (-4)} \oplus \overline{ \rhhpitchcupout (-3)} \oplus \overline{\rhunit \bhunit \runit (-3)} \oplus \overline{\bhunit \rhunit (-4)}  \oplus \overline{\rhunit \bhunit (-4)}\oplus \overline{ \bhhpitchcupout (-3)} \oplus \overline{\bhunit \rhunit \bunit (-3)}\oplus \overline{\bhunit \rhunit (-4)} \oplus \overline{\bhunit\ \raisebox{-1ex}{ \dbox{\scalemath{0.9}{\alpha_s^\vee}}} (-5)} \\
&\scalemath{0.84}{\underset{\hh^2(d_{|\beta|-3})}{
\xrightarrow{
\begin{pmatrix}
\alpha_s & 1 & 0 & -1 & 1 & 0 & 0 & 0 & 0 &0 \\
    0 & 0 & 1 & \alpha_s & 0 & 0 & 1 & \alpha_t & 0 &0\\
    0 & 0 & 1 & \alpha_s &0 & 0 &1 &\alpha_t & 0 & 0 \\
    0 & 0 & 0 & 0 & 0 &1 & 0 & -1 & 1 & \alpha_t
\end{pmatrix}}
} \overline{\rhunit\ \raisebox{-1ex}{ \dbox{\scalemath{0.9}{\alpha_t^\vee}}} (-3)}\oplus  \overline{\rhunit \bhunit (-2)} \oplus \overline{\bhunit \rhunit (-2)} \oplus  \overline{\bhunit\ \raisebox{-1ex}{ \dbox{\scalemath{0.9}{\alpha_s^\vee}}} (-3)}  \underset{\hh^2(d_{|\beta|-2})}{
\xrightarrow{
\begin{pmatrix}
    0 & -1 & 1 & 0  \\
    0 & 1 & -1 & 0 
\end{pmatrix}}}\overline{\rhunit\ \raisebox{-1ex}{ \dbox{\scalemath{0.9}{\alpha_t^\vee}}} (-3)} \oplus \overline{\bhunit\ \raisebox{-1ex}{ \dbox{\scalemath{0.9}{\alpha_s^\vee}}} (-1)}} 
\end{align*}
As right $R-$modules $ \paren{\overline{\hhh}^{A=2}}^{T=|\beta|-2}=0$ with image generated by $\vec{w_2}, \vec{w_3}, \vec{w_6}$. \\

Our computations above show that for each $i$, $\im \hh^i(d_{|\beta|-3})$ is spanned by $\hh^i(d_{|\beta|-3})$ applied to 
$$\hh^i(C_{sst}),\qquad \hh^i(C_{sts}),\qquad \hh^i(C_{stt}), \qquad  \hh^i(C_{tst})$$   
and these all show up if $\beta$ contains $stst$ as a subexpression. (Note that $\im \hh^i(C_{sst})=\im \hh^i(C_{tss})$ and $\im \hh^i(C_{stt}) =\im \hh^i(C_{tts}) $.) Similarly with $tsts$. Finally, note that $\ker \hh^i(d_{|\beta|-2})$ doesn't change if we only have $stst$ or $tsts$ instead of $ststst$. 
\end{proof}

\subsection{Negative braids}

\cref{lem:dualcomplex} for $n=3$ says that
\begin{equation}
\hh^k(F^\bullet_{\beta^\vee}) \cong \un{\hom}_R^\bullet(\hh^{2-k}(F^\bullet_{\beta}), R) (4)     
\end{equation}

So applying this to the complexes in the previous section, we obtain the analogue of \cref{theorem:hhhpos2}. 

\begin{theorem}
\label{theorem:hhhneg2}
Let $\alpha$ be \un{any} negative braid on 3 strands containing $(stst)^{-1}$ or $(tsts)^{-1}$ as a subexpression. Then
\[ \paren{\overline{\hhh}(\alpha)^{A=0}}^{T=-|\alpha|+2}=  0, \qquad \paren{\overline{\hhh}(\alpha)^{A=1}}^{T=-|\alpha|+2}= 0, \qquad \paren{\overline{\hhh}(\alpha)^{A=2}}^{T=-|\alpha|+2}=  \kb(8-|\alpha|)  \]
\end{theorem}
\begin{proof}
    The kernels are fairly easy to work out in these cases as many columns in the matrices in the previous subsection only have a single entry in them. We will just note the corresponding generator for the only nonzero cohomology group above is given by
    \[ \kb\paren{\runit\bunit -\bunit\runit  (|\alpha|-2)}^\vee (4) \]
\end{proof}

\begin{remark}
By \cref{2lowerlem} and \cref{raslemma}, \cref{theorem:hhhpos2} and \cref{theorem:hhhneg2} is sufficient to compute $T-$degree $|\beta|-2$, ($-|\alpha|+2$) for any positive (negative) braid $\beta$. 
\end{remark}

\section{$T-$degree $|\beta|-2$ for general $n$}
\label{gensect}

When $n>3$, for purposes of calculating $\overline{\hhh}(\widehat{\beta})$ for $\beta\in B_n^+$, \cref{raslemma} and \cref{2lowerlem} still allows us to assume that either $\sigma_{i}\sigma_{i+1}\sigma_i \sigma_{i+1}$ or $\sigma_{i+1}\sigma_{i}\sigma_{i+1} \sigma_{i}$ appears as a subexpression in $\beta$ for all $1\le i \le n-2$. However, we may not necessarily have the subexpressions $susu$ or $usus$ when $s,u$ are not adjacent. Fortunately, \cref{1lowerlem} still applies and so $\beta$ contains one of $ssuu, susu, suus, uuss, usus, ussu$. Another thing to note is that $R_\beta^{|\beta|-3}$ now contains terms of the form $C_{stu}=B_sB_tB_u$ where $s\neq t\neq u$. The corresponding $R-$basis can be found in \cref{basisappend}.

\subsection{Case 1}

Assume that for every pair $i\neq j\in [n-1]$, $\beta\in B_n^+$ contains either  $\sigma_i \sigma_j \sigma_i \sigma_j$ or $ \sigma_j \sigma_i \sigma_j \sigma_i$ as a subexpression.

\subsubsection{$\ker \hh^\bullet(d_{|\beta|-2})$}
Given our assumptions above, $d_{|\beta|-2}$ in $ R_\beta^\bullet$ will be of the form

\begin{equation}
\label{rbetaeq2}
   \oplus_{i=1}^{n-1}C_{ii}(|\beta|-2) \oplus \overbrace{B_1B_2(|\beta|-2)\oplus B_2B_1(|\beta|-2)\cdots \oplus B_{n-2}B_{n-1}(|\beta|-2)}^{S^{|\beta|-2}_\beta}\underset{d_{|\beta|-2}}{\xrightarrow{ 
   \lcounit{i} \lzero{j}  -\lzero{i} \lcounit{j} 
   }} \overbrace{B_1(|\beta|-1)\oplus \cdots \oplus B_{n-1}(|\beta|-1)}^{R^{|\beta|-1}_\beta} 
\end{equation} 

where $S_\beta^{|\beta|-2}$ now contains both $C_{ij}$ and $C_{ji}$. Because $\hh^k(d_{|\beta|-2})|_{C_{ss}}=\raisebox{1ex}{\ds}$, by the diagrammatics in \cref{basisappend}, it follows that $\hh^k(d_{|\beta|-2})|_{C_{ss}}=0$ and so all elements of $\hh^k(C_{ii})$ will be in $\ker \hh^\bullet(d_{|\beta|-2})$. By \cref{fjgjlem2}, the rest of the kernel breaks up into blocks indexed by $J\subseteq[n-1]$. 

\begin{example}
\label{ex2}
Take $n=4$, so $R=\kb[\alpha_1, \alpha_2, \alpha_3]$. Take $J=\set{1,2}\subset \set{1,2,3}$. The differential from $F_J$ to $G_J$ of the corresponding $J$ block of the matrix $\hh^k(d_{|\beta|-2})$ (modulo $Q$ grading shifts and $R$) will be
\[ \scalemath{0.9}{\lhunit{2}\lhunit{1}\oplus \lhunit{1}\lhunit{2}\oplus \lunit{3}\lhunit{1} \scalemath{0.85}{\dboxed{\alpha_2^\vee}} \oplus \lhunit{1}\lunit{3} \scalemath{0.85}{\dboxed{\alpha_2^\vee}} \oplus \lunit{3}\lhunit{2} \scalemath{0.85}{\dboxed{\alpha_1^\vee}} \oplus \lhunit{2}\lunit{3} \scalemath{0.85}{\dboxed{\alpha_1^\vee}} \ \raisebox{-3.5ex}{$\underset{\hh^2(d_{|\beta|-2})|_{F_J}}{\xrightarrow{\begin{pmatrix}
    1 & -1 & \alpha_3 & -\alpha_3 & 0 & 0 \\
    -1 & 1 & 0 & 0 & \alpha_3 & -\alpha_3 \\
    0 & 0 & -1 & 1 & -1 & 1
\end{pmatrix}}}$} \lhunit{1} \scalemath{0.85}{\dboxed{\alpha_2^\vee}}\oplus \lhunit{2} \scalemath{0.85}{\dboxed{\alpha_1^\vee}} \oplus \lunit{3} \scalemath{0.85}{\dboxed{\alpha_1^\vee\alpha_2^\vee}}} \]
It's clear that $[ 1 \ 1 \ 0 \ 0 \ 0 \ 0 ]^T, [  0 \ 0  \ 1 \ 1 \ 0 \ 0 ]^T, [  0 \ 0  \ 0 \ 0 \ 1 \ 1 ]^T $ are elements of the kernel. Because $\hh^k(d_{|\beta|-3})(\lhunit{1}\lhunit{2}\lunit{3})\in F_J$, this is also in the kernel. We claim these generate the kernel. Let $M$ be the matrix above and let
\[ P=\scalemath{0.9}{\begin{pmatrix}
    1 & 0 & 0 & \alpha_3 & 0 & 0 \\
    1 & 0 & 0 & 0 & 1 & 0 \\
    0 & 1 & 0 & -1 & 0 & 0 \\
    0 & 1 & 0 & 0 & 0 & 1 \\
    0 & 0 & 1 & 1 & 0 & 0 \\
    0 & 0 & 1 & 0 & 0 & 0
\end{pmatrix}}\implies MP=\begin{pmatrix}
    0 & 0 & 0 & 0 & -1 &-\alpha_3 \\
    0 & 0 & 0 & 0 & 1 & 0 \\
    0 & 0 & 0 & 0 & 0 & 1
\end{pmatrix}  \]
$P$ is invertible and the rightmost $3\times 2$ submatrix of $MP$ has no nontrivial solutions and this proves the claim. 
\end{example}

\begin{lemma}
\label{any3lem}
Let \lhalfunit{i}=\lunit{i} or \lhunit{i} and let $w\in S_3$ act on $\lhalfunit{s} \lhalfunit{t} \lhalfunit{u}$ by permuting each \lhalfunit{i}. For $s,t,u$ all pairwise not equal, 
    \[\hh^\bullet(d_{|\beta|-3})\paren{ w(\lhalfunit{s} \lhalfunit{t} \lhalfunit{u} )}\in \hh^\bullet(d_{|\beta|-3})\paren{ \lhalfunit{s} \lhalfunit{t} \lhalfunit{u} } R\bigoplus_{j\neq k\neq \ell\in \set{s,t,u}} \raisebox{-3.5ex}{\begin{tikzpicture}[scale=0.6]
    \draw (0,0.1)-- (0,0.6);
    \draw[thick] (0,0) circle (1.3mm);
    
    \fill[black] (0,0.15) arc[start angle=90, end angle=270, radius=1.3mm] -- cycle;
    \node at (0,-0.5) {$j$};
    \node[dot] at (0,0.6) {};
\end{tikzpicture}} \paren{\lhalfunit{k} \lhalfunit{\ell} +\lhalfunit{\ell} \lhalfunit{k}} \ R \]
\end{lemma}
\begin{proof}
    The assumptions on $s,t,u$ means that $\hh^\bullet(d_{|\beta|-3})=\counit \  \0 \  \0  -\0 \ \counit \ \0 + \0 \  \0 \ \counit$. Let us consider the case $w=(12)$, the rest are similar. Then
     \[\hh^\bullet(d_{|\beta|-3})\paren{ \lhalfunit{t} \lhalfunit{s} \lhalfunit{u} }=-\hh^\bullet(d_{|\beta|-3})\paren{ \lhalfunit{s} \lhalfunit{t} \lhalfunit{u} } + \raisebox{-3.5ex}{\begin{tikzpicture}[scale=0.6]
    \draw (0,0.1)-- (0,0.6);
    \draw[thick] (0,0) circle (1.3mm);
    
    \fill[black] (0,0.15) arc[start angle=90, end angle=270, radius=1.3mm] -- cycle;
    \node at (0,-0.5) {$u$};
    \node[dot] at (0,0.6) {};
\end{tikzpicture}} \paren{\lhalfunit{s} \lhalfunit{t} +\lhalfunit{t} \lhalfunit{s}}  \]
\end{proof}

Note that $\lhalfunit{s} \lhalfunit{t} \lhalfunit{u}$ is fixed in the lemma above. For example, $\hh^\bullet(d_{|\beta|-3})\paren{ \bhunit \rhunit \gunit }\in \hh^\bullet(d_{|\beta|-3})\paren{ \rhunit \bhunit \gunit} R \oplus\ldots$, but $\hh^\bullet(d_{|\beta|-3})\paren{ \bhunit \rhunit \gunit }\not\in \hh^\bullet(d_{|\beta|-3})\paren{ \rhunit \bunit \ghunit} R \oplus\ldots$

\begin{lemma}
\label{fjgjlem2}
    For $J\subseteq [n-1]$ and $|J|=k$ let
\[ F_{J}=\bigoplus_{\substack{i_0,j_0\in J \\ i_0\neq j_0}} \lhunit{i_0} \lhunit{j_0} \scalemath{0.85}{\dboxed{\widehat{\alpha_J^{i_0, j_0}}^\vee}}R\oplus\bigoplus_{\substack{i_0\in J\\j\not\in J}} \lunit{j}\lhunit{i_0} \scalemath{0.85}{\dboxed{\widehat{\alpha_J^{i_0}}^\vee}} R\oplus\bigoplus_{\substack{i_0\in J\\j\not\in J}}\lhunit{i_0}  \lunit{j} \scalemath{0.85}{\dboxed{\widehat{\alpha_J^{i_0}}^\vee}} R\oplus \bigoplus_{\substack{i,j\not\in J\\ i\neq j }} \lunit{i}\lunit{j}\scalemath{0.85}{\dboxed{\alpha_{J}^\vee}} R       \]
    Then $\displaystyle \hh^k(S^{|\beta|-2}_\beta)=\bigoplus_{|J|=k} F_J $ and $\hh^k(d_{|\beta|-2})\paren{F_{J}}\subseteq G_{J}$ $\implies \hh^k(d_{|\beta|-2})$ is a block matrix.
\end{lemma}

\begin{lemma}
\label{ejlem}
For $J\subseteq [n-1]$ and $|J|=k$, we have $\hh^k(d_{|\beta|-3})(E_J)\subseteq F_J$ where
    \[ E_{J}= \bigoplus_{\substack{i_0,j_0, k_0\in J \\ i_0\neq j_0\neq k_0}} \lhunit{i_0} \lhunit{j_0} \lhunit{k_0} \scalemath{0.85}{\dboxed{\widehat{\alpha_J^{i_0, j_0, k_0}}^\vee}} \,R \oplus \bigoplus_{\substack{i_0\neq j_0\in J \\ k\not\in J}} \sbrac{\lhunit{i_0} \lhunit{j_0} \lunit{k}}_{S_3} \scalemath{0.85}{\dboxed{\widehat{\alpha_J^{i_0, j_0}}^\vee}}R\oplus\bigoplus_{\substack{i_0\in J\\j\neq k\not\in J}} \sbrac{\lhunit{i_0} \lunit{j} \lunit{k}}_{S_3} \scalemath{0.85}{\dboxed{\widehat{\alpha_J^{i_0}}^\vee}} R\oplus \bigoplus_{\substack{i,j,k\not\in J\\ i\neq j\neq k }} \lunit{i}\lunit{j} \lunit{k}\scalemath{0.85}{\dboxed{\alpha_{J}^\vee}} R\]
    where $\sbrac{abc}_{S_3}=\oplus_{w\in S_3} w(abc)$ and $F_J$ is as in \cref{fjgjlem2}.
\end{lemma}

\begin{prop}
\label{fjker}
Suppose that $\beta\in B_n^+$ contains the subexpressions $\sigma_i\sigma_j$ and $\sigma_j \sigma_i$ for any $i,j\in [n-1]$. Fix $J\subseteq [n-1]$. Then
 $$\ker \hh^\bullet(d_{|\beta|-2})|_{F_J}=\hh^\bullet(d_{|\beta|-3})(E_J)\oplus\bigoplus_{\substack{i_0,j_0\in J \\ i_0\neq j_0}} \sbrac{\lhunit{i_0} \lhunit{j_0}}_{+} \scalemath{0.85}{\dboxed{\widehat{\alpha_J^{i_0, j_0}}^\vee}}R\oplus\bigoplus_{\substack{i_0\in J\\j\not\in J}} \sbrac{\lunit{j}\lhunit{i_0}}_{+} \scalemath{0.85}{\dboxed{\widehat{\alpha_J^{i_0}}^\vee}} R\oplus \bigoplus_{\substack{i,j\not\in J\\ i\neq j }} \sbrac{\lunit{i}\lunit{j}}_{+}\scalemath{0.85}{\dboxed{\alpha_{J}^\vee}} R $$ 
 where $\sbrac{ab}_+=ab+ba$.
\end{prop}
\begin{proof}
For $a\neq b\in [n-1]$, let $D(a)=\lhunit{a}$ if $a\in J$ and $D(a)=\lunit{a}$ if $a\not\in J$. Then because $F_J$ now contains both $D(a)D(b)$ and $D(b)D(a)$, it follows that $D(a)D(b)+D(b)D(a)$\footnote{We have omitted the exterior box for readability purposes.} are in $\ker \hh^\bullet(d_{|\beta|-2})|_{F_J}$ which gives the rightmost part after the first $\oplus$ of our claimed decomposition above. Proceeding somewhat similarly to \cref{ex2}, let $M$ be the matrix representing $\hh^\bullet(d_{|\beta|-2})|_{F_J}$ in the basis as seen in \cref{fjgjlem2} and consider the invertible change of basis matrix $P$
\[P=\paren{\begin{array}{ccc|c c c}
    1 & 0 & \cdots & 0 & 0  & \cdots\\
    1 & 0 & \cdots & 1 & 0 & \cdots \\
    0 & 1 & \cdots & 0 & 0 & \cdots\\
    0 & 1 & \cdots & 0 & 1 & \cdots \\
\vdots& \vdots &  \ddots& \vdots & \vdots &\ddots
\end{array}} \qquad \qquad P_6=\paren{\begin{array}{ccc|c c c}
    1 & 0 & 0 & 0 & 0  & 0\\
    1 & 0 & 0 & 1 & 0 & 0 \\
    0 & 1 & 0 & 0 & 0 & 0\\
    0 & 1 & 0 & 0 & 1 & 0 \\
    0 & 0 &  1 & 0 & 0 &0 \\
     0 & 0 &  1 & 0 & 0 &1
\end{array}}\]
where each column in the first half has exactly 2 1's in adjacent rows (aka the basis vector $D(a)D(b)+D(b)D(a)$, and each column in the second half has exactly one 1 in the even rows. $P_6$ is shown on the right above in case things aren't clear. As a result $MP=\sbrac{M_1|M_2}$ where $M_1$ is the zero matrix while $M_2$ will precisely be the matrix $\hh^k(d^\prime_{|\beta|-2})|_{F_J^\prime}$ from \cref{betabeta1sect} (e.g, the even columns of the matrix in \cref{ex2} is the matrix in \cref{ex1}) which from the proof of \cref{mainpostheorem} is $d_2$ in some Koszul complex \cref{koszuleq} whose cohomology vanishes at $d_2$. Thus $\ker M_2=\im d_3$ which one can see is exactly $\hh^\bullet(d_{|\beta|-3})(E_J)$. Note, not all terms of $E_J$ necessarily appear in $R_\beta^{|\beta|-3}$, only one in each $S_3$ orbit. But by \cref{any3lem} this is sufficient. Informally \cref{any3lem} is saying that $E_J``="\Lambda^3(\oplus_{a\in J}D(a))$.
\end{proof}

\begin{prop}
\label{kerprop}
$\ker \hh^k(d_{|\beta|-2})=\bigoplus_{|J|=k} \mathscr{C}_J$ where\footnote{We have omitted the $Q$ grading shift in the first three $R-$submodules as they cancel out in cohomology anyways}
\[ \mathscr{C}_J=\hh^\bullet(d_{|\beta|-3})(E_J) \oplus N_J \oplus T_J \oplus\bigoplus_{j\not\in J}\lunit{j}\scalemath{0.85}{\dboxed{\alpha_{J}^\vee}} \, R(|\beta|-3)\oplus \bigoplus_{i_0\in J}\lhunit{i_0} \scalemath{0.85}{\dboxed{\widehat{\alpha_J^{i_0}}^\vee}} \, R (|\beta|-3)
\]
\begin{enumerate}[(1)]
    \item $\displaystyle N_J=\bigoplus_{\substack{i_0,j_0\in J \\ |i_0- j_0|=1}} \sbrac{\lhunit{i_0} \lhunit{j_0}}_+ \scalemath{0.85}{\dboxed{\widehat{\alpha_J^{i_0, j_0}}^\vee}}R\oplus\bigoplus_{\substack{i_0\in J, j\not\in J\\ |i_0-j|=1}} \paren{\sbrac{\lunit{j}\lhunit{i_0}}_+ +\lhunit{i_0}} \scalemath{0.85}{\dboxed{\widehat{\alpha_J^{i_0}}^\vee}} R\oplus \bigoplus_{\substack{i,j\not\in J\\ i=j -1}} \paren{\sbrac{\lunit{i}\lunit{j}}_+ +\lunit{i} }\scalemath{0.85}{\dboxed{\alpha_{J}^\vee}} R   $
    \item $\displaystyle T_J=\bigoplus_{\substack{i_0,j_0\in J \\ |i_0- j_0|>1}} \sbrac{\lhunit{i_0} \lhunit{j_0}}_{+} \scalemath{0.85}{\dboxed{\widehat{\alpha_J^{i_0, j_0}}^\vee}}R\oplus\bigoplus_{\substack{i_0\in J, j\not\in J\\ |i_0-j|>1}} \sbrac{\lunit{j}\lhunit{i_0}}_{+} \scalemath{0.85}{\dboxed{\widehat{\alpha_J^{i_0}}^\vee}} R\oplus \bigoplus_{\substack{i, j\not\in J\\ |i-j|>1}}  \sbrac{\lunit{i}\lunit{j}}_{+}\scalemath{0.85}{\dboxed{\alpha_{J}^\vee}} R$
\end{enumerate}
\end{prop}
\begin{proof}
    We have just added the elements of $\hh^k(C_{ii})=\hh^k(B_i(-1))$ (which recall are all in the kernel) to \cref{fjker} and separated out the non $\hh^\bullet(d_{|\beta|-3})(E_J)$ part of \cref{fjker} into two parts: $N_J$ where $i,j$ are close together and $T_J$ where $i,j$ are far apart. In addition for $N_J$ we have performed a change of basis whose matrix is clearly lower triangular for $N_J$ so that it will cancel out with certain image vectors in the next subsection.
\end{proof}

\subsubsection{Cohomology at $\hh^\bullet(d_{|\beta|-2})$}

\begin{theorem}
\label{prebeta2theorem}
    Let $\beta\in B_n^+$ be \un{any} positive braid on $n$ strands s.t. for every pair $i\neq j\in [n-1]$, $\beta$ contains either  $\sigma_i \sigma_j \sigma_i \sigma_j$ or $ \sigma_j \sigma_i \sigma_j \sigma_i$ as a subexpression. Then 
\[ \paren{\overline{\hhh}(\beta)^{A=k}}^{T=|\beta|-2}= 
\begin{cases}
 \kb(|\beta|-4) & k=0 \\
 \kb(|\beta|) & k=1 \\
  0 & k\ge2
\end{cases}
  \]
\end{theorem}
\begin{proof}
$(1)$ Clearly $\hh^\bullet(d_{|\beta|-3})(E_J)$ cancels. \\
$(2)$ From the explicit matrices of \cref{theorem:hhhpos2}, it follows that for $s=\sigma_i ,t=\sigma_{i+1}$, $ 1\le i\le n-1 $, the image of the differential $\hh^\bullet(d_{|\beta|-3})$  contains the following vectors below ($\vec{w_4}$ in $\overline{\hhh}^{A=0}$,  $ \vec{w_6}, \vec{w_{14}} $ in $\overline{\hhh}^{A=1}$,$\vec{w_3}$ in $\overline{\hhh}^{A=2}$ )
 \[ d_{tst} \paren{\bpitchcup}=\runit\bunit+\bunit\runit+\bunit, \quad d_{sts}\paren{\rhpitchcupout }=\rhunit+\rhunit\bunit +\bunit \rhunit, \quad d_{tst}(\bhpitchcupout)=\runit \bhunit +\bhunit\runit+\bhunit, \quad d_{sts}\paren{\rhhpitchcupout}=\rhunit\bhunit+\bhunit\rhunit \]
 where the subscript is used to denote where the basis element comes from, e.g. $d_{sts}:C_{sts}\to \ldots$. As exterior boxes can be ``taken out" of $d_{sts}$ this means that $N_J$ cancels.\\
 $(3)$ As shown in \cref{basisappend}, the bases for $\hh^i(B)$ appearing in the last 3 terms of $R^\bullet_\beta$ don't change even if $m_{su}=2$. However the differentials might, as 2$-$color Ext Soergel relations change when $m_{su}=2$. Specifically, the pitchfork differential applied to the ``cup basis vectors" are now 0. We record the analogues of the matrices in the proof of \cref{theorem:hhhpos2} below. \\ 
$\bullet$ Differentials for $\overline{\hhh}^{A=0}$ when $m_{su}=2$:
    \[  \hh^0(d_{|\beta|-3})=\scalemath{0.95}{\begin{pNiceMatrix}[
  first-row,code-for-first-row=\scriptstyle,
  last-col,code-for-last-col=\scriptstyle,
]
        C_{sss}& C_{ssu} & C_{sus} & C_{sus} & C_{uss} & C_{suu} & C_{usu} & C_{usu} & C_{uus} & C_{uuu} & \\
        \alpha_s & \alpha_u & -\alpha_u & 0& \alpha_u & 0 & 0 & 0 & 0& 0  & C_{ss}\\
        0 & 0 & \alpha_s & 1 & 0 & 0 & \alpha_u & 1 & 0  & 0 & C_{su}\\
        0 & 0 & \alpha_s & 1 & 0 & 0 & \alpha_u & 1 & 0 & 0  & C_{us}\\
        0 & 0 & 0  & 0 & 0 & \alpha_s & -\alpha_s & 0 & \alpha_s & \alpha_u & C_{uu}
    \end{pNiceMatrix}}\]

    $ \bullet$ Differentials for $\overline{\hhh}^{A=1}$ when $m_{su}=2$: 
    \[ \hh^1(d_{|\beta|-3})=\scalemath{0.85}{\begin{pNiceMatrix}[
  first-row,code-for-first-row=\scriptstyle,
  last-col,code-for-last-col=\scriptstyle,
]
    C_{sss}&C_{sss} &C_{ssu} &C_{ssu} & C_{sus} & C_{sus} & C_{sus}& C_{sus} & C_{uss} & C_{uss} & C_{suu} & C_{suu} & C_{usu} & C_{usu} & C_{usu} & C_{usu} & C_{uus} & C_{uus} & C_{uuu} &C_{uuu} & \\
    \alpha_s & 0 & \alpha_u & 0 & -\alpha_u & 0 & 0 & 0 & 0 &\alpha_u & 0 & 0 & 0 & 0 & 0 & 0 & 0 &0 & 0 &0  & C_{ss}\\
    0 & \alpha_s & 0 & 1 & 0 & 0 &0 & -1 & 1 & 0 & 0 & 0 & 0 & 0 &0 &0 & 0 & 0 & 0 &0 & C_{ss} \\
    0 & 0 & 0 & 0 &\alpha_s & 1 & 0 & 0 & 0 & 0 & 0 & 0 & 0 & 0 & 1 & \alpha_u &0 &0 & 0 & 0  & C_{su}\\
    0 & 0 &0 &0 &0 &0 & 1 & \alpha_s &0 & 0 & 0 & 0 &\alpha_u & 1 & 0 & 0 &0 &0 &0 & 0  & C_{su}\\
    0 & 0 &0 &0 &0 &0 & 1 & \alpha_s & 0 & 0 &0 &0 &\alpha_u & 1 & 0 &0 &0 &0 &0 &0 & C_{us} \\
    0 & 0 &0 &0 &\alpha_s &1 & 0 & 0 & 0 & 0& 0 &0 &0 &0 & 1 & \alpha_u & 0 &0 &0 &0 & C_{us} \\
    0 & 0 &0 &0 &0 &0 & 0 & 0 & 0 & 0& 0 &\alpha_s & -\alpha_s &0 &0 & 0 & \alpha_s & 0 &\alpha_u &0 & C_{uu} \\
    0 & 0 &0 &0 &0 &0 & 0 & 0 & 0 & 0& 1 & 0 & 0 & 0 & 0 & -1 & 0 & 1 & 0 &\alpha_u & C_{uu}
    \end{pNiceMatrix}
    }   \]
$ \bullet$ Differentials for $\overline{\hhh}^{A=2}$ when $m_{su}=2$: 

\[  \hh^2(d_{|\beta|-3})=\begin{pNiceMatrix}[
  first-row,code-for-first-row=\scriptstyle,
  last-col,code-for-last-col=\scriptstyle,
]
C_{sss}& C_{ssu} & C_{sus} & C_{sus} & C_{uss} & C_{suu} & C_{usu} & C_{usu} & C_{uus} & C_{uuu} & \\
\alpha_s & 1 & 0 & -1 & 1 & 0 & 0 & 0 & 0 &0 & C_{ss} \\
    0 & 0 & 1 & \alpha_s & 0 & 0 & 1 & \alpha_u & 0 &0  & C_{su}\\
    0 & 0 & 1 & \alpha_s &0 & 0 &1 &\alpha_u & 0 & 0 & C_{us} \\
    0 & 0 & 0 & 0 & 0 &1 & 0 & -1 & 1 & \alpha_u &  C_{uu}
\end{pNiceMatrix}  \]
Thus when $s,u$ are far apart, the image of the differential $\hh^\bullet(d_{|\beta|-3})$  contains the following vectors below ($\vec{w_4}$ in $\overline{\hhh}^{A=0}$,  $ \vec{w_6}, \vec{w_{7}} $ in $\overline{\hhh}^{A=1}$,$\vec{w_3}$ in $\overline{\hhh}^{A=2}$ 
 \begin{equation}
 \label{suimage}
     d_{sus} \paren{\gpitchcup}=\runit\gunit+\gunit\runit, \quad d_{sus}\paren{\grhpitchcupout }=\rhunit\gunit +\gunit \rhunit, \quad d_{sus}(\ghpitchcupout)=\runit \ghunit +\ghunit\runit, \quad d_{sus}\paren{\grhhpitchcupout}=\rhunit\ghunit+\ghunit\rhunit 
 \end{equation} 
from which it follows that $T_J$ cancels.\\
$(4)$ Suppose that $|J|\ge 1$. Then for any $j\not\in J$ and any $k\in J$
\[ d_{jj k}\paren{\lunit{j} \lhunit{k} \scalemath{0.85}{\dboxed{\widehat{\alpha_J^{k}}^\vee}}} = \lunit{j} \scalemath{0.85}{\dboxed{\alpha_J^\vee}}\]
Otherwise we must have $J=\emptyset$ (which is all of $A=0$). From the matrix for $\overline{\hhh}^{A=0}$ in the proof of \cref{theorem:hhhpos2}, when $s,t$ are adjacent, the only other elements mapping to $\lunit{i}$ are of the form \vspace{-2ex}
$$d_{sst}\paren{ \lunit{s} \lunit{t}} = \alpha_t \lunit{s},\qquad  d_{stt}\paren{ \lunit{s} \lunit{t}} = \alpha_s \lunit{t}$$
The same is true when $t$ is replaced by $u$ where $s,u$ are not adjacent. But then as we range over all $s,t$ adjacent, the corresponding image encompasses the cases where $s,u$ are not adjacent. Thus the cohomology here at $A=0$, $T=|\beta|-2$ is $\runit (|\beta|-3) \, R/(\alpha_1, \ldots, \alpha_{n-1})\cong \kb(|\beta|-4)$. \\

$(5)$ Suppose that $|J|\ge 2$. Then for any $i_0\in J$, we can find another $j_0\in J$ s.t. $i_0\neq j_0$. Then 
\[ d_{i_0 i_0 j_0}\paren{\lhunit{i_0} \lhunit{j_0} \scalemath{0.85}{\dboxed{\widehat{\alpha_J^{i_0, j_0}}^\vee}}} = \lhunit{i_0} \scalemath{0.85}{\dboxed{\widehat{\alpha_J^{i_0}}^\vee}}\]
Otherwise suppose $|J|=1$,  while for $s,t$ adjacent
\[d_{sts}\paren{ \rhpitchcupin } -d_{tst}\paren{ \bhpitchcupout }= \rhunit+ \runit\bhunit+\bhunit\runit-\paren{\runit \bhunit+\bhunit \runit +\bhunit}=\rhunit-\bhunit \]
and thus all the $\lhunit{i}$ for $1\le i \le n-1$ are equivalent in cohomology. From the matrix for $\overline{\hhh}^{A=1}$ in the proof of \cref{theorem:hhhpos2}, when $s,t$ are adjacent, the only other elements mapping to $\lhunit{i}$ are of the form \vspace{-2ex}
$$d_{sst}\paren{ \lhunit{s} \lunit{t}} = \alpha_t \lhunit{s},\qquad  d_{stt}\paren{ \lunit{s} \lhunit{t}} = \alpha_s \lhunit{t}$$
The same is true when $t$ is replaced by $u$ where $s,u$ are not adjacent. But then as we range over all $s,t$ adjacent the image encompasses the cases where $s,u$ are not adjacent. From $(4)$, all the $\lunit{j}  \scalemath{0.85}{\dboxed{\alpha_k}}$ terms die in cohomology. Thus the cohomology here at $A=1$, $T=|\beta|-2$ is $\rhunit (|\beta|-3)\, R/(\alpha_1, \ldots, \alpha_{n-1})\cong \kb(|\beta|)$.
\end{proof}

\subsection{The Other Cases}

As mentioned at the start of \cref{gensect} we may not necessarily have the subexpressions $susu$ or $usus$ when $s,u$ are not adjacent. Suppose for a fixed $s,u$ we only have the subexpression $suus$ or $ussu$. $\ker \hh^\bullet(d_{|\beta|-2})$ is not affected, as the conditions in \cref{fjker} are still satisfied. Notice we still have the subexpression $sus$ or $usu$. The only time we need the image from far away subexpressions in the proof of \cref{prebeta2theorem} is in step $(3)$ and \cref{suimage} tells us we just need the vectors coming from $C_{sus}$ which also equals the vectors coming from $C_{usu}$. Thus the cohomology at $T=|\beta|-2$ is still the same.  \\

Finally suppose for a fixed $s,u$ we only have the subexpressions $ssuu$ (equivalently $uuss$). Then the only change in \cref{kerprop} is that we no longer have $D(s)D(u)+D(u)D(s)$ appearing in the $T_J$ part. As for the image of $\hh^k(d_{|\beta|-3})$ there is now no subexpression $sus$ or $usu$ and the rest of the image vectors are not needed in the proof of \cref{prebeta2theorem} and thus the cohomology is the same as before.

\begin{example}
\label{ex3}
    Consider \cref{ex2} where $n=4$ and $J=\set{1,2}\subset \set{1,2,3}$. Suppose that between $1$ and $3$, only the subexpression $1133$ appears. Then the differential from $F_J$ to $G_J$ of the corresponding $J$ block of the matrix $\hh^k(d_{|\beta|-2})$ (modulo $Q$ grading shifts and $R$) will be
\[ \scalemath{0.9}{\lhunit{2}\lhunit{1}\oplus \lhunit{1}\lhunit{2} \oplus \lhunit{1}\lunit{3} \scalemath{0.85}{\dboxed{\alpha_2^\vee}} \oplus \lunit{3}\lhunit{2} \scalemath{0.85}{\dboxed{\alpha_1^\vee}} \oplus \lhunit{2}\lunit{3} \scalemath{0.85}{\dboxed{\alpha_1^\vee}} \ \raisebox{-3.5ex}{$\underset{\hh^2(d_{|\beta|-2})|_{F_J}}{\xrightarrow{\begin{pmatrix}
    1 & -1  & -\alpha_3 & 0 & 0 \\
    -1 & 1  & 0 & \alpha_3 & -\alpha_3 \\
    0 & 0  & 1 & -1 & 1
\end{pmatrix}}}$} \lhunit{1} \scalemath{0.85}{\dboxed{\alpha_2^\vee}}\oplus \lhunit{2} \scalemath{0.85}{\dboxed{\alpha_1^\vee}} \oplus \lunit{3} \scalemath{0.85}{\dboxed{\alpha_1^\vee\alpha_2^\vee}}} \]  
Applying the change of basis 
\[ P=\scalemath{0.9}{\begin{pmatrix}
    1 & 0 & 0 &0 & 0  \\
    1 & 0 & 1 & 0 & 0  \\
    0 & 0 & 0 & 1 & 0 \\ 
    0 & 1 & 0  & 0 & 0 \\
    0 & 1 & 0  & 0 & 1 \\
\end{pmatrix}} \]
we obtain the kernel from \cref{ex2} except without $\lunit{3}\lhunit{1} \scalemath{0.85}{\dboxed{\alpha_2^\vee}} + \lhunit{1}\lunit{3} \scalemath{0.85}{\dboxed{\alpha_2^\vee}}$.
\end{example}

To summarize, we have obtained the following theorem.

\begin{maintheorem}
\label{mainpostheorem2}
    Let $\beta\in B_n^+$ be \un{any} positive braid on $n$ strands s.t. $\beta$ contains either  $\sigma_i \sigma_{i+1} \sigma_i \sigma_{i+1}$ or $ \sigma_{i+1} \sigma_i \sigma_{i+1} \sigma_i$ as a subexpression for all $1\le i \le n-1$. Then as $R-$modules
\[ \paren{\overline{\hhh}(\beta)^{A=k}}^{T=|\beta|-2}= 
\begin{cases}
 \kb(|\beta|-4) & k=0 \\
 \kb(|\beta|) & k=1 \\
  0 & k\ge2
\end{cases}
  \]
\end{maintheorem}

\subsection{Negative Braids}

Like before, we will apply \cref{lem:dualcomplex} to compute for negative braids. However, unlike for positive braids where the cohomology in the $n-$strand case turned out to be exactly as in the $3-$strand case, the cohomology for $n\ge 4$ at $T=-|\alpha|+2$ will be different than for $n=3$. Informally, this is because $E_J$ in \cref{ejlem} only appears in $R^{|\beta|-3}_\beta$ when $n\ge 4$ and $E_J$ is what makes computations Koszul. In the proof below please try to think in the ``transpose."

\begin{maintheorem}
\label{mainnegtheorem2}
    Let $\alpha\in B_n^-$ be \un{any} negative braid on $n\ge 4$ strands s.t. $\alpha$ contains either  $\paren{\sigma_i \sigma_{i+1} \sigma_i \sigma_{i+1}}^{-1}$ or $ \paren{\sigma_{i+1} \sigma_i \sigma_{i+1} \sigma_i}^{-1}$ as a subexpression for all $1\le i \le n-1$. Then as $R-$modules
\[ \paren{\overline{\hhh}(\alpha)^{A=k}}^{T=-|\alpha|+2}= 0 \quad \forall k\ge 0
  \]
\end{maintheorem}
\begin{proof}
$(i)$ Set $\beta=\alpha^{-1}$. The assumption on the subexpressions implies that $d_{|\beta|-3}$ in $R^{\bullet}_\beta$ is of the form\footnote{we have omitted $Q$ grading shifts since cohomology turns out to be 0 anyways.}
\[\bigoplus_{\substack{i=1 \\ |i-j|=1 }}^{n-1}C_{iij}\oplus\bigoplus_{\substack{i=1 \\ |i-j|=1 }}^{n-1}C_{ijj}\oplus \paren{\text{ rest of }R^{|\beta|-3}_\beta }\underset{d_{|\beta|-3}}{\xrightarrow{\begin{pmatrix}
    \lzero{i} \lcounit{j} & \lcounit{i} \lzero{j}   & \ldots \\
    0 & 0 & \ldots
\end{pmatrix}}
}\bigoplus_{i=1}^{n-1}C_{ii}\oplus \paren{\text{ rest of }R^{|\beta|-2}_\beta }  \]
where we have only shown the differential on the $C_{iij}$ and $C_{ijj}$ parts. As a result, we see that
\[  d_{iij}\paren{ \lhalfunit{i} \lunit{j} \scalemath{0.85}{\dboxed{\alpha_{J}^\vee}}} = \alpha_j \lhalfunit{i} \scalemath{0.85}{\dboxed{\alpha_{J}^\vee}}\]
As $\displaystyle \hh^k\paren{\bigoplus_{i=1}^{n-1}C_{ii}  }=\bigoplus_{|J|=k} G_J$(see \cref{gjlem}), the above implies that $\hh^k\paren{\bigoplus_{i=1}^{n-1}C_{ii}  }\cap \ker \hh^k\paren{d_{|\beta|-3}}^T=\set{0}$. \\

$(ii)$ From $(i)$ we can remove all the $C_{ii}$ terms. Thus, from the proof of \cref{prebeta2theorem}, $\im \hh^k\paren{d_{|\beta|-3}}$ is now generated by elements of the form \cref{suimage} times exterior boxes (where $s,u$ can also be adjacent) and $\hh^k(d_{|\beta|-3})(E_J)$.

and as a result $\hh^k\paren{d_{|\beta|-3}}^T|_{\text{ rest of }R^{|\beta|-2}_\beta}$ is a block matrix, again indexed by $J\subseteq [n-1]$. Each block $J$ can be further decomposed into a block matrix which we will first illustrate with an example.

\begin{example}
\label{ex4}
    Let $n=4$ and $J=\set{1,2}\subset \set{1,2,3}$. Let red $=1$, blue $=2$, green $=3$. Suppose that between $1$ and $3$, only the subexpression $1133$ appears. Then $\im \hh^k(d_{|\beta|-3})$ in the $J$ block is of the form (ommitting $Q$ grading shifts and $R$) 
    \[ \rhhpitchcupout \oplus \bgrhpitchcupout \scalemath{0.85}{\dboxed{\alpha_{1}^\vee}} \oplus \rhunit\bhunit \gunit \underset{\hh^2(d_{|\beta|-3})}{\xrightarrow{\begin{pNiceArray}{cc|c}[margin]
    1  & 0 & \alpha_3   \\
    0 & 1 & \alpha_1 \\
     0  & 0  & -\alpha_2 \\ \hline
     1 & 0 & 0 \\
     0 & 1 & 0
\end{pNiceArray}}} \rhunit \bhunit \oplus \bhunit \gunit \scalemath{0.85}{\dboxed{\alpha_{1}^\vee}} \oplus \rhunit \gunit \scalemath{0.85}{\dboxed{\alpha_{2}^\vee}}\oplus \bhunit\rhunit \oplus \gunit\bhunit  \scalemath{0.85}{\dboxed{\alpha_{1}^\vee}} \]
\end{example}
In general for each pair $i<j$ define $P(i,j)=ij$ if $ij$ appears as a subexpression of $\beta$ and $ji$ otherwise. Recall that if $i$ and $j$ are non-adjacent, we might only have one of $ij$ or $ji$ appearing in $\beta$. Let $\ell$ denote the number of such non-adjacent pairs where only one ordering occurs and $N=\binom{n-1}{2}-\ell$. Then the $J$ block of $\hh^k(d_{|\beta|-3})$ is of the form
\[\begin{pNiceArray}{ccc|c}[margin]
   \Block{2-3}<\LARGE>{I_{N}}& & & \Block{3-1}{d_3(K_J)}  \\
    & & & \\
    \Block{1-3}{0_{\ell\times N}} &  & &  \\ \hline
     \Block{2-3}<\LARGE>{I_{N} } &&& \Block{2-1}<\LARGE>{\vec{0}} \\
      &  & &  \\
\end{pNiceArray}\]

\begin{itemize}
    \item where $d_3(K_J)$ be the differential $d_3$ in the Koszul complex $K_J=K_\bullet(\set{\alpha_i }_{i\not\in J}, \overbrace{ 1, \ldots, 1}^{|J| \ many})$.
    \item The first $N+\ell=\binom{n}{2}$ rows correspond to the basis vectors $\lhalfunit{} \lhalfunit{}$ with labelling $P(i,j)$. The last $N$ rows are again the basis vectors $\lhalfunit{} \lhalfunit{}$ but with labelling the flip of $P(i,j)$, which appears as both $ij$ and $ji$ are subexpressions of $\beta$.
    \item The first $N$ columns correspond to the cups from \cref{suimage}. The last $\binom{n-1}{3}$ columns correspond to $\Lambda^3K_J$ where we have used \cref{any3lem} to ensure the image under lands in the first $\binom{n}{2}$ rows.
\end{itemize}
As before \cref{koszulvanish} will imply $d_3(K_J)^T$ is exact, and this determines the rest of the kernel of $\hh^k(d_{|\beta|-3})^T$ in the $J$ block. It is then clear that this is the image of $\hh^k(d_{|\beta|-2})^T$, see \cref{rbetaeq2}. For example, transposing the matrix in \cref{ex3} will give the kernel in \cref{ex4}.
\end{proof}

\appendix
\begin{appendices}

\section{Bases for Hochschild Cohomology}
\label{basisappend}

\subsection{Geometric Realization of Rank $2$}
Here $R=\kb[\alpha_s, \alpha_t]$. $s$ will be color red and $t$ will be color blue. Let $m_{st}$ denote the order of $st$.

\begin{lemma}
\label{hhbslem}
\begin{align*}
\hh^{0}(B_s)&=  \runit \, R=R(-1)  & \hh^{0}(B_t)&=  \bunit \, R=R(-1)  \\
\hh^{1}(B_s)&=  \rhunit  \, R\oplus \runit\ \raisebox{-1ex}{ \dboxed{\scalemath{0.9}{\alpha_t^\vee}}}  \, R =R(3)\oplus R(1) & \hh^{1}(B_t)&=  \bhunit  \, R\oplus \bunit\ \raisebox{-1ex}{ \dboxed{\scalemath{0.9}{\alpha_s^\vee}}}  \, R =R(3)\oplus R(1) \\
\hh^{2}(B_s)&=  \rhunit\ \raisebox{-1ex}{ \dboxed{\scalemath{0.9}{\alpha_t^\vee}}}  \, R =R(5) & \hh^{2}(B_t)&=  \bhunit\ \raisebox{-1ex}{ \dboxed{\scalemath{0.9}{\alpha_s^\vee}}}  \, R =R(5)
\end{align*}
\end{lemma}

\begin{lemma}
\begin{align*}
\hh^{0}(R)&=  \raisebox{-1ex}{ \dboxed{\scalemath{0.9}{1}}} R=R &\hh^{0}(B_sB_t)&=  \runit \bunit \, R =R(-2) \\
\hh^{1}(R)&=  \raisebox{-1ex}{ \dboxed{\scalemath{0.9}{\alpha_s^\vee}}} R \oplus  \raisebox{-1ex}{ \dboxed{\scalemath{0.9}{\alpha_t^\vee}}} R =R(2)\oplus R(2) &\hh^{1}(B_sB_t)&=  \rhunit \bunit \, R\oplus  \runit \bhunit \, R=R(2)\oplus R(2)  \\
\hh^{2}(R)&=  \raisebox{-1ex}{ \dboxed{\scalemath{0.9}{\alpha_s^\vee}}}\raisebox{-1ex}{ \dboxed{\scalemath{0.9}{\alpha_t^\vee}}} \, R=R(4)&\hh^{2}(B_sB_t)&=  \rhunit \bhunit \, R=R(6)
\end{align*}
\end{lemma}

\begin{lemma}
\label{hhsts}
For any $s\neq t$
\begin{align*}
 \hh^{0}(B_tB_sB_t)&=  \bunit \runit \bunit \, R \oplus \bpitchcup \, R   =R(-3)\oplus R(-1) \\
\hh^{1}(B_tB_sB_t)&=  \bhunit \runit \bunit \, R  \oplus \bhpitchcupout \, R  \oplus \bhpitchcupin \, R  \oplus  \bunit \rhunit \bunit \, R= R(1)\oplus R(3)\oplus R(3)\oplus R(1) \\
\hh^{2}(B_tB_sB_t)&= \bhhpitchcupout \, R \oplus \bhunit \rhunit \bunit \, R=R(7)\oplus R(5)
\end{align*}
\end{lemma}
\begin{proof}
    The diagrams for $\hh^0$ is precisely the double leaves basis for $\hom_{R^e}(R, B_tB_sB_t)$ for any $s,t$, regardless of $m_{st}$. \un{Let us first do the case $m_{st}=2$.} To use \cite[Theorem 7.6]{Li22} for the higher $\hh^i$, first note by adjunction that
    \[ \ext^0_{R^e}(B_t, B_s B_t)= \substack{ \runit \bunit\\ \bcounit } \ R \oplus \raisebox{0ex}{\runit} \bzero \ R \]
    Now, \cite[Theorem 7.6]{Li22} says that the $\ker \rho_s^e(\un{st}) \un{\rho_t t(\rho_t)}$ part of $\ext^1_{R^e}(B_t, B_s B_t)$ is gotten by precomposing the above by \bhzero. This results in
    \[ \ker \rho_s^e(\un{ts}) \textnormal{ part of } \ext^1_{R^e}(B_t, B_s B_t)= \substack{ \runit \bunit\\ \bhcounit } \ R \oplus \raisebox{0.3ex}{\runit} \bhzero \ R \]
    Now for the $\Db(\ker \rho_s^e(\un{st}))\underline{\rho_s}$ part, we need to first calculate $\Db(\ker \rho_s^e(\un{st}))$. In this case, the Light Leaves basis actually equals the $01$ basis, i.e. the first two columns of the below table are equal in $\bs(\un{st})$, and similarly the last two columns are equal.
   { \begin{center}
    {
\setlength{\extrarowheight}{8pt}
        \begin{tabular}{c|c|c|c  }
           LL  &$01$ basis form & $\Db$ of 01 & LL of $\Db$  \\\hline
        $\substack{ \runit \bunit\\ \phantom{\bcounit} }$    & $\substack{ \runit \bunit\\ \rcounit \bcounit }$  & \raisebox{0.5ex}{\scalemath{1.3}{ \rzero \ \bzero}} & \raisebox{0.5ex}{\scalemath{1.3}{ \rzero \ \bzero}} \\[2ex] \hline
        \raisebox{1ex}{\runit \ \scalemath{1.2}{\bzero}} &  $\substack{ \runit \\ \rcounit} \raisebox{0.5ex}{\scalemath{1.4}{\bzero}} $ & \raisebox{0.5ex}{\scalemath{1.4}{\rzero}}$\substack{ \bunit \\ \bcounit}  $ &\raisebox{0.5ex}{\scalemath{1.4}{\rzero}}$\substack{ \bunit \\ \phantom{\bcounit}}  $
        \end{tabular}
       }
    \end{center}}
    We now need to precompose the rightmost column with the corresponding $\sbrac{{}_{tt}^0\psi_s^{1(\un{w})}}$, which can be found in the discussion following the proof of \cite[Theorem 7.6]{Li22}. We then obtain
    \[ \Db(\ker \rho_s^e(\un{st})) \textnormal{ part of } \ext^1_{R^e}(B_t, B_s B_t)= \substack{ \rhunit \bunit\\ \bcounit } \ R \oplus \raisebox{0ex}{\rhunit} \bzero \ R \]
    For $\ext^2_{R^e}(B_t, B_s B_t)$, \cite[Theorem 7.6]{Li22} says we just need to precompose the above diagrams with \bhzero. Finally, rotating all these diagrams up and to the left gives us the desired basis for $\ext^\bullet_{R^e}(R, B_t B_s B_t)$.\\
    
    \un{When $m_{st}\neq 2$} we follow the same procedure except when calculating $\Db(\ker \rho_s^e(\un{st}))\underline{\rho_s}$, we instead precompose by 
    \[ \begin{tikzpicture}[scale=0.7]
	       \draw[red] (1.1,-0.1) -- (2,-1);
	       \draw[blue] (2,0.1) -- (2,-1);
	       \draw[red] (2.9,-0.1) -- (2,-1);
	       \draw[blue] (2,-2) -- (2,-1);
	       \node[rkhdot] at (2,-1) {};
	       \node at (2,-1) {$\s{4}$};
	       \node[rdot] at (2.9,-0.1) {};
	\end{tikzpicture}\raisebox{4ex}{=} \ \  \begin{tikzpicture}[scale=0.6]
	       \draw[red] (1.1,0.1) -- (1.1,-1);
	       \draw[blue] (2,0.1) -- (2,-1);
	       \draw[blue] (2,-2) -- (2,-1);
	       \node[rhdot] at (1.1,-1) {};
	\end{tikzpicture}\raisebox{0.6cm}{$\ \ - \ \ $}
\begin{tikzpicture}[scale=0.6]
	       \draw[red] (1.1,0.1) -- (1.1,-1);
	       \draw[blue] (2,0.1) -- (2,-1);
	       \draw[blue] (2,-2) -- (2,-1);   
	       \node[rdot] at (1.1,-1) {};
	       \node[bhdot] at (2,-1.3) {};
	\end{tikzpicture}, \qquad \qquad
    \begin{tikzpicture}[scale=0.7]
	       \draw[red] (2,0) -- (2,-1);
	       \draw[blue] (2,-2) -- (2,-1);
	       \node[rkhdot] at (2,-1) {};
	       \node at (2,-1) {$\s{2}$};
	\end{tikzpicture} \ \ \raisebox{4ex}{=}
    \begin{tikzpicture}[scale=0.7]
    \draw[red] (1.1,-0.1) -- (2,-1);
	       \draw[blue] (2,0.1) -- (2,-1);
	       \draw[red] (2.9,-0.1) -- (2,-1);
	       \draw[blue] (2,-2) -- (2,-1);
	       \node[rkhdot] at (2,-1) {};
	       \node at (2,-1) {$\s{4}$};
	       \node[rdot] at (2.9,-0.1) {};
           \node[bdot] at (2,0.1) {};
	\end{tikzpicture}\raisebox{4ex}{=} \ \ \begin{tikzpicture}[scale=0.7]
	       \draw[red] (2,0) -- (2,-0.8);
	       \draw[blue] (2,-2) -- (2,-1.2);
	       \node[rhdot] at (2,-0.8) {};
           \node[bdot] at (2,-1.2) {};
	\end{tikzpicture}\raisebox{4ex}{$\ \ - \ \ $} \begin{tikzpicture}[scale=0.7]
	       \draw[red] (2,0) -- (2,-0.8);
	       \draw[blue] (2,-2) -- (2,-1.2);
	       \node[rdot] at (2,-0.8) {};
           \node[bhdot] at (2,-1.2) {};
	\end{tikzpicture}\]
where we have used 4-Ext Reduction on the LHS above. But as remarked after the proof of \cite[Theorem 3.13]{Li22} we can add any element of $\ker \rho_s^e(\un{ts})$ and the result will still be a basis. In particular, we can get rid of the second picture after the $-$ sign in each diagram above and so we end up with the same result as in the $m_{st}=2$ case.
\end{proof}

\subsection{Arbitrary Realization}

For the geometric realization of $S_n$, $R=R_n=\kb[\alpha_1, \ldots, \alpha_{n-1}]$. For $M\in R_n^e-\mathrm{mod}$, let $\hh^k_n(M):=\ext^k_{R_n^e}(R_n, M)$. Let $\Lambda =\Lambda^\bullet[\alpha_1^\vee, \ldots, \alpha_{n-1}^\vee]$. Then \cite[Lemma 3.1]{Li22} says that $\hh^\bullet_{R_{n}}(B_i)=\widetilde{\hh^\bullet_{R_2}}(B_i)\otimes_\kb \Lambda/(\alpha_i^\vee)$ and for $i\neq j$
$$\hh^\bullet_{R_{n}}(B_iB_j)=\widetilde{\hh^\bullet_{R_3}}(B_i B_j)\otimes_\kb \Lambda/(\alpha_i^\vee, \alpha_j^\vee), \qquad \hh^\bullet_{R_{n}}(B_iB_jB_i)=\widetilde{\hh^\bullet_{R_3}}(B_i B_j B_i)\otimes_\kb \Lambda/(\alpha_i^\vee, \alpha_j^\vee) $$
where $\widetilde{\hh^\bullet_{R_2}}(B_i)=\hh^\bullet_{R_2}(B_i)\otimes_{R_2}R_n$ (see \cref{hhbslem} for the case $n=3$) and similarly with $\widetilde{\hh^\bullet_{R_3}}(B_i B_j)$. There is a slight abuse of notation as we actually have $R_2=\kb[\alpha_i], R_3=\kb[\alpha_i, \alpha_j]$. Diagrammatically, the exterior algebra parts above are encoded in exterior boxes. For example, when $n=3$, $\Lambda=\Lambda^\bullet[\alpha_s^\vee, \alpha_t^\vee]$ and $\hh^\bullet_{R_{3}}(B_s), \hh^\bullet_{R_{3}}(B_t)$ is precisely \cref{hhbslem}. Finally, we also need

\begin{lemma}
For $s,t,u$ all pairwise not equal, $R_4=\kb[\alpha_s, \alpha_t, \alpha_u]$ and $u$ green, $\hh^\bullet_{R_{n}}(B_sB_tB_u)=\widetilde{\hh^\bullet_{R_4}}(B_s B_t B_u)\otimes_\kb \Lambda/(\alpha_s^\vee, \alpha_t^\vee, \alpha_u^\vee )$ and 
\begin{align*}
 &\hh^{0}_{R_4}(B_sB_tB_u)=  \runit \bunit \gunit \, R_4 =R_4(-3) \\
&\hh^{1}_{R_4}(B_sB_tB_u)=  \rhunit \bunit \gunit \, R_4\oplus  \runit \bhunit \gunit \, R_4\oplus  \runit \bunit \ghunit \, R_4=R_4(1)\oplus R_4(1)\oplus R_4(1)  \\
&\hh^{2}_{R_4}(B_sB_tB_u)=  \rhunit \bhunit \gunit  \, R_4\oplus \rhunit \bunit \ghunit  \, R_4\oplus \runit \bhunit \ghunit  \, R_4=R_4(5)\oplus R_4(5) \oplus R_4(5) \\
&\hh^{3}_{R_4}(B_sB_tB_u)=  \rhunit \bhunit \ghunit  \, R_4=R_4(9)
\end{align*}
\end{lemma}

\end{appendices}

\printbibliography

@article {BM90,
    AUTHOR = {Birman, Joan S. and Menasco, William W.},
     TITLE = {Studying links via closed braids. {IV}. {C}omposite links and
              split links},
   JOURNAL = {Invent. Math.},
  FJOURNAL = {Inventiones Mathematicae},
    VOLUME = {102},
      YEAR = {1990},
    NUMBER = {1},
     PAGES = {115--139},
      ISSN = {0020-9910,1432-1297},
   MRCLASS = {57M25 (20F36)},
  MRNUMBER = {1069243},
MRREVIEWER = {Hugh\ Reynolds\ Morton},
       DOI = {10.1007/BF01233423},
       URL = {https://doi.org/10.1007/BF01233423},
}

@misc{Rou04,
      title={Categorification of the braid groups}, 
      author={Raphael Rouquier},
      year={2004},
      eprint={math/0409593},
      archivePrefix={arXiv},
      primaryClass={math.RT},
      url={https://arxiv.org/abs/math/0409593}, 
}

@article {Hog18,
    AUTHOR = {Hogancamp, Matthew},
     TITLE = {Categorified {Y}oung symmetrizers and stable homology of torus
              links},
   JOURNAL = {Geom. Topol.},
  FJOURNAL = {Geometry \& Topology},
    VOLUME = {22},
      YEAR = {2018},
    NUMBER = {5},
     PAGES = {2943--3002},
      ISSN = {1465-3060,1364-0380},
   MRCLASS = {18G60 (57M27)},
  MRNUMBER = {3811775},
MRREVIEWER = {Benjamin\ Cooper},
       DOI = {10.2140/gt.2018.22.2943},
       URL = {https://doi.org/10.2140/gt.2018.22.2943},
}

@article {Kho00,
    AUTHOR = {Khovanov, Mikhail},
     TITLE = {A categorification of the {J}ones polynomial},
   JOURNAL = {Duke Math. J.},
  FJOURNAL = {Duke Mathematical Journal},
    VOLUME = {101},
      YEAR = {2000},
    NUMBER = {3},
     PAGES = {359--426},
      ISSN = {0012-7094,1547-7398},
   MRCLASS = {57M27 (57R56)},
  MRNUMBER = {1740682},
       DOI = {10.1215/S0012-7094-00-10131-7},
       URL = {https://doi.org/10.1215/S0012-7094-00-10131-7},
}

@online{M,
  author       = {Makisumi, Shotaro},
  title        = {Diagrammatics for Ext-Enhanced Soergel Bimodules in Type $A_1$},
  date         = {2022},
  eprint       = {https://makisumi.com/math/articles/hsbim.pdf},
}

@article {KoszulAMRW,
    AUTHOR = {Achar, Pramod N. and Makisumi, Shotaro and Riche, Simon and
              Williamson, Geordie},
     TITLE = {Koszul duality for {K}ac-{M}oody groups and characters of
              tilting modules},
   JOURNAL = {J. Amer. Math. Soc.},
  FJOURNAL = {Journal of the American Mathematical Society},
    VOLUME = {32},
      YEAR = {2019},
    NUMBER = {1},
     PAGES = {261--310},
      ISSN = {0894-0347},
   MRCLASS = {20G05 (16D80 16S37 20G44)},
  MRNUMBER = {3868004},
       DOI = {10.1090/jams/905},
       URL = {https://doi.org/10.1090/jams/905},
}

@article {Soe90,
    AUTHOR = {Soergel, Wolfgang},
     TITLE = {Kategorie {$\mathscr{O}$}, perverse {G}arben und {M}oduln \"{u}ber den
              {K}oinvarianten zur {W}eylgruppe},
   JOURNAL = {J. Amer. Math. Soc.},
  FJOURNAL = {Journal of the American Mathematical Society},
    VOLUME = {3},
      YEAR = {1990},
    NUMBER = {2},
     PAGES = {421--445},
      ISSN = {0894-0347},
   MRCLASS = {17B35},
  MRNUMBER = {1029692},
MRREVIEWER = {Ian M. Musson},
       DOI = {10.2307/1990960},
       URL = {https://doi-org.ezproxy.cul.columbia.edu/10.2307/1990960},
}

@article {EK10,
    AUTHOR = {Elias, Ben and Khovanov, Mikhail},
     TITLE = {Diagrammatics for {S}oergel categories},
   JOURNAL = {Int. J. Math. Math. Sci.},
  FJOURNAL = {International Journal of Mathematics and Mathematical
              Sciences},
      YEAR = {2010},
     PAGES = {Art. ID 978635, 58},
      ISSN = {0161-1712},
   MRCLASS = {18D10 (17B10 20C08)},
  MRNUMBER = {3095655},
MRREVIEWER = {Theo Johnson-Freyd},
       DOI = {10.1155/2010/978635},
       URL = {https://doi-org.ezproxy.cul.columbia.edu/10.1155/2010/978635},
}

@article {DC,
    AUTHOR = {Elias, Ben},
     TITLE = {The two-color {S}oergel calculus},
   JOURNAL = {Compos. Math.},
  FJOURNAL = {Compositio Mathematica},
    VOLUME = {152},
      YEAR = {2016},
    NUMBER = {2},
     PAGES = {327--398},
      ISSN = {0010-437X},
   MRCLASS = {20C08 (18D05)},
  MRNUMBER = {3462556},
MRREVIEWER = {Vanessa Miemietz},
       DOI = {10.1112/S0010437X15007587},
       URL = {https://doi-org.ezproxy.cul.columbia.edu/10.1112/S0010437X15007587},
}

@article {SC,
    AUTHOR = {Elias, Ben and Williamson, Geordie},
     TITLE = {Soergel calculus},
   JOURNAL = {Represent. Theory},
  FJOURNAL = {Representation Theory. An Electronic Journal of the American
              Mathematical Society},
    VOLUME = {20},
      YEAR = {2016},
     PAGES = {295--374},
   MRCLASS = {20C08 (18D10 20F55)},
  MRNUMBER = {3555156},
MRREVIEWER = {Vanessa Miemietz},
       DOI = {10.1090/ert/481},
       URL = {https://doi-org.ezproxy.cul.columbia.edu/10.1090/ert/481},
}

@article {HodgeSoergel,
    AUTHOR = {Elias, Ben and Williamson, Geordie},
     TITLE = {The {H}odge theory of {S}oergel bimodules},
   JOURNAL = {Ann. of Math. (2)},
  FJOURNAL = {Annals of Mathematics. Second Series},
    VOLUME = {180},
      YEAR = {2014},
    NUMBER = {3},
     PAGES = {1089--1136},
      ISSN = {0003-486X},
   MRCLASS = {20C08 (20F55)},
  MRNUMBER = {3245013},
MRREVIEWER = {Chi Kin Mak},
       DOI = {10.4007/annals.2014.180.3.6},
       URL = {https://doi-org.ezproxy.cul.columbia.edu/10.4007/annals.2014.180.3.6},
}

@article {Wtorsionexplod,
    AUTHOR = {Williamson, Geordie},
     TITLE = {Schubert calculus and torsion explosion},
      NOTE = {With a joint appendix with Alex Kontorovich and Peter J.
              McNamara},
   JOURNAL = {J. Amer. Math. Soc.},
  FJOURNAL = {Journal of the American Mathematical Society},
    VOLUME = {30},
      YEAR = {2017},
    NUMBER = {4},
     PAGES = {1023--1046},
      ISSN = {0894-0347},
   MRCLASS = {20C20 (14M15 14N15 20G05)},
  MRNUMBER = {3671935},
MRREVIEWER = {Volodymyr Mazorchuk},
       DOI = {10.1090/jams/868},
       URL = {https://doi-org.ezproxy.cul.columbia.edu/10.1090/jams/868},
}

@article {Pic20,
    AUTHOR = {Piccirillo, Lisa},
     TITLE = {The {C}onway knot is not slice},
   JOURNAL = {Ann. of Math. (2)},
  FJOURNAL = {Annals of Mathematics. Second Series},
    VOLUME = {191},
      YEAR = {2020},
    NUMBER = {2},
     PAGES = {581--591},
      ISSN = {0003-486X,1939-8980},
   MRCLASS = {57K10 (57R65)},
  MRNUMBER = {4076631},
MRREVIEWER = {Laurence\ R.\ Taylor},
       DOI = {10.4007/annals.2020.191.2.5},
       URL = {https://doi.org/10.4007/annals.2020.191.2.5},
}

@article {Ras10,
    AUTHOR = {Rasmussen, Jacob},
     TITLE = {Khovanov homology and the slice genus},
   JOURNAL = {Invent. Math.},
  FJOURNAL = {Inventiones Mathematicae},
    VOLUME = {182},
      YEAR = {2010},
    NUMBER = {2},
     PAGES = {419--447},
      ISSN = {0020-9910,1432-1297},
   MRCLASS = {57M27},
  MRNUMBER = {2729272},
MRREVIEWER = {William\ D.\ Gillam},
       DOI = {10.1007/s00222-010-0275-6},
       URL = {https://doi.org/10.1007/s00222-010-0275-6},
}

@article {Kho07,
    AUTHOR = {Khovanov, Mikhail},
     TITLE = {Triply-graded link homology and {H}ochschild homology of
              {S}oergel bimodules},
   JOURNAL = {Internat. J. Math.},
  FJOURNAL = {International Journal of Mathematics},
    VOLUME = {18},
      YEAR = {2007},
    NUMBER = {8},
     PAGES = {869--885},
      ISSN = {0129-167X},
   MRCLASS = {57M27 (18G60)},
  MRNUMBER = {2339573},
MRREVIEWER = {Sabin Cautis},
       DOI = {10.1142/S0129167X07004400},
       URL = {https://doi-org.ezproxy.cul.columbia.edu/10.1142/S0129167X07004400},
}

@online{HM19,
  author       = {Hogancamp, Matthew and Mellit, Anton},
  title        = {Torus link homology},
  date         = {2019},
  eprinttype   = {arxiv},
  eprint       = {1909.00418},
}

@article {EH19,
    AUTHOR = {Elias, Ben and Hogancamp, Matthew},
     TITLE = {On the computation of torus link homology},
   JOURNAL = {Compos. Math.},
  FJOURNAL = {Compositio Mathematica},
    VOLUME = {155},
      YEAR = {2019},
    NUMBER = {1},
     PAGES = {164--205},
      ISSN = {0010-437X},
   MRCLASS = {57M27 (57M25)},
  MRNUMBER = {3880028},
MRREVIEWER = {Mark Clifford Hughes},
       DOI = {10.1112/s0010437x18007571},
       URL = {https://doi.org/10.1112/s0010437x18007571},
}

@article {Ras15,
    AUTHOR = {Rasmussen, Jacob},
     TITLE = {Some differentials on {K}hovanov-{R}ozansky homology},
   JOURNAL = {Geom. Topol.},
  FJOURNAL = {Geometry \& Topology},
    VOLUME = {19},
      YEAR = {2015},
    NUMBER = {6},
     PAGES = {3031--3104},
      ISSN = {1465-3060},
   MRCLASS = {57M27 (55T25)},
  MRNUMBER = {3447099},
MRREVIEWER = {Matthew Hogancamp},
}

@article {BCH23,
    AUTHOR = {Bowman, Chris and Cox, Anton and Hazi, Amit},
     TITLE = {Path isomorphisms between quiver {H}ecke and diagrammatic
              {B}ott-{S}amelson endomorphism algebras},
   JOURNAL = {Adv. Math.},
  FJOURNAL = {Advances in Mathematics},
    VOLUME = {429},
      YEAR = {2023},
      ISSN = {0001-8708,1090-2082},
   MRCLASS = {20C08 (20G05)},
  MRNUMBER = {4611117},
MRREVIEWER = {Stuart\ Martin},
       DOI = {10.1016/j.aim.2023.109185},
       URL = {https://doi.org/10.1016/j.aim.2023.109185},
}

@article {EQ23,
    AUTHOR = {Elias, Ben and Qi, You},
     TITLE = {Categorifying {H}ecke algebras at prime roots of unity, part
              {I}},
   JOURNAL = {Trans. Amer. Math. Soc.},
  FJOURNAL = {Transactions of the American Mathematical Society},
    VOLUME = {376},
      YEAR = {2023},
    NUMBER = {11},
     PAGES = {7691--7742},
      ISSN = {0002-9947,1088-6850},
   MRCLASS = {18N25 (20C08)},
  MRNUMBER = {4657219},
MRREVIEWER = {Matthew\ Westaway},
       DOI = {10.1090/tran/8908},
       URL = {https://doi.org/10.1090/tran/8908},
}

@article {LW22,
    AUTHOR = {Libedinsky, Nicolas and Williamson, Geordie},
     TITLE = {The anti-spherical category},
   JOURNAL = {Adv. Math.},
  FJOURNAL = {Advances in Mathematics},
    VOLUME = {405},
      YEAR = {2022},
      ISSN = {0001-8708,1090-2082},
   MRCLASS = {20C08 (05E10 18B99 20C20 20G40 20J99)},
  MRNUMBER = {4437613},
MRREVIEWER = {Shoumin\ Liu},
       DOI = {10.1016/j.aim.2022.108509},
       URL = {https://doi.org/10.1016/j.aim.2022.108509},
}

@article {Mal24,
    AUTHOR = {Maltoni, Leonardo},
     TITLE = {Reducing {R}ouquier complexes},
   JOURNAL = {Proc. Lond. Math. Soc. (3)},
  FJOURNAL = {Proceedings of the London Mathematical Society. Third Series},
    VOLUME = {129},
      YEAR = {2024},
    NUMBER = {1},
      ISSN = {0024-6115,1460-244X},
   MRCLASS = {20G05 (13P20 18N25 20F36)},
  MRNUMBER = {4766986},
}

@article {GHM19,
    AUTHOR = {Gorsky, Eugene and Hogancamp, Matthew and Mellit, Anton and
              Nakagane, Keita},
     TITLE = {Serre duality for {K}hovanov-{R}ozansky homology},
   JOURNAL = {Selecta Math. (N.S.)},
  FJOURNAL = {Selecta Mathematica. New Series},
    VOLUME = {25},
      YEAR = {2019},
    NUMBER = {5},
      ISSN = {1022-1824,1420-9020},
   MRCLASS = {57K18 (18G80 20C08)},
  MRNUMBER = {4036505},
MRREVIEWER = {Pedro\ Vaz},
       DOI = {10.1007/s00029-019-0524-5},
       URL = {https://doi.org/10.1007/s00029-019-0524-5},
}

@misc{NS24,
      title={Computations of HOMFLY homology}, 
      author={Keita Nakagane and Taketo Sano},
      year={2024},
      eprint={2111.00388},
      archivePrefix={arXiv},
      primaryClass={math.GT},
      url={https://arxiv.org/abs/2111.00388}, 
}

@misc{Li22,
      title={The Two-Color Ext Soergel Calculus}, 
      author={Cailan Li},
      year={2022},
      eprint={2211.07802},
      archivePrefix={arXiv},
      primaryClass={math.RT},
      url={https://arxiv.org/abs/2211.07802}, 
}

\end{document}